\documentclass[11pt,reqno]{article}

\usepackage[utf8]{inputenc}
\usepackage[T1]{fontenc}
\usepackage{lmodern}
\usepackage{amsmath,amsthm,amsfonts,amssymb}
\usepackage{geometry}
\usepackage{multicol}
\usepackage{enumitem}
\setlist{itemsep=2pt,topsep=4pt,parsep=0pt}
\usepackage{mathrsfs}
\usepackage{graphicx}
\usepackage{hyperref}
\hypersetup{colorlinks=true,linkcolor=blue,citecolor=blue,urlcolor=blue}

\newtheorem{mainthm}{Theorem}

\newtheorem{maincor}{Corollary}

\newtheorem{theorem}{Theorem}[section]
\newtheorem{lemma}[theorem]{Lemma}
\newtheorem{corollary}[theorem]{Corollary}
\newtheorem{proposition}[theorem]{Proposition}
\theoremstyle{definition}

\newtheorem{exe}[theorem]{Example}
\theoremstyle{remark}
\newtheorem{remark}[theorem]{Remark}

\newcommand{\C}{\mathbb C}
\newcommand{\R}{\mathbb R}
\newcommand{\Z}{\mathbb Z}
\newcommand{\PP}{\mathbb P}

\newcommand{\F}{\mathcal F}

\newcommand{\T}{\mathbb T}
\newcommand{\pd}{\partial}
\newcommand{\ol}{\overline}
\newcommand{\ip}[2]{\left\langle #1,#2\right\rangle}

\title{Resonant saddles, tangencies with small spheres,\\ and the characteristic flow}
\author{Toshikazu Ito\thanks{Professor Emeritus, Ryukoku University, 67 Tsukamoto-cho, Fukakusa, Fushimi-ku, Kyoto 612-8577, Japan.} \and Bruno Sc\'ardua\thanks{Instituto de Matem\'atica, Universidade Federal do Rio de Janeiro, Av. Athos da Silveira Ramos 149, Cidade Universit\'aria, Rio de Janeiro, RJ 21941-909, Brazil. Corresponding author: \texttt{bruno.scardua@gmail.com}.}}
\date{}

\begin{document}
\maketitle

\begin{abstract}
We study resonant Siegel saddles in $\C^2$, with eigenvalue ratio $-p/q$, where $p,q$ are coprime positive integers, through their position relative to small round spheres. An intrinsic invariant $k$, defined through a resonance index that measures approximate first integrals, equals the first resonant obstruction. For $k<\infty$, the maximal contact order with the linear foliation is $m=k(p+q)+1$, and suitably normalized charts give $2k$ tangency circles. We compute their transverse type: resonant terms cancel to first order, and Poincar\'e--Dulac normalization through degree $2m-1$ makes every contact a saddle, while weaker normalization can produce elliptic arcs and periodic characteristic orbits. For the exact model we describe the whole-sphere flow and the correspondence between its tangency circles and separatrix sectors. In arbitrary charts the characteristic return map near a separatrix circle realizes its holonomy, whose $p$-th iterate has, when $k<\infty$, multiplicity $kp+1$, with $kp$ attracting and $kp$ repelling petals. Classical rigidity of resonant holonomy then recovers the analytic saddle class from $C^1$ orbital equivalence near that circle. We also describe tangency loci for arbitrary nondegenerate germs by the numerical range and the leading orbital form.
\end{abstract}


{\small\setcounter{tocdepth}{2}\setlength{\parskip}{0pt}
\begin{multicols}{2}\tableofcontents\end{multicols}}

\section{Introduction}
This work concerns the analytic classification of singularities of holomorphic vector fields in dimension two. Let $X\in\mathfrak X(\C^2,0)$ be a germ of holomorphic vector field with a nondegenerate singularity at $0$, i.e.\ with $DX(0)$ invertible, let $\lambda_1,\lambda_2\in\C\setminus\{0\}$ be the eigenvalues of $DX(0)$, and let $\mathcal C\subset\C$ be the convex hull of $\{\lambda_1,\lambda_2\}$. There are two cases.

(i) \emph{The Poincar\'e case}, $0\notin\mathcal C$. The analytic theory is given by the Poincar\'e--Dulac normal form \cite{Dulac}, \cite[Theorem~4.10]{Ilyashenko-Yakovenko}, \cite{Arnold-GM}. Transversality of $X$ to all small spheres centred at the origin forces this case \cite{Ito,Ito1996}.

(ii) \emph{The Siegel case}, $0\in\mathcal C$, i.e.\ $\lambda_1/\lambda_2\in\R_{<0}$. We then say that $X$, or the foliation it defines, has a \emph{Siegel saddle} at $0$. By the Briot--Bouquet theorem it has exactly two separatrices, smooth and tangent to the eigendirections. Here $X$ is tangent to every small sphere centred at the origin \cite{ItoScardua2005}, and the analytic classification is much richer \cite{MatteiMoussu,MartinetRamis}: it is equivalent to the analytic classification of the holonomy of a separatrix, a germ $f(z)=az+\sum_{j\ge2}a_jz^j\in\operatorname{Diff}(\C,0)$ with $a=e^{2\pi i\lambda_1/\lambda_2}$ \cite{MatteiMoussu,MartinetRamis,perezmarcoyoccoz}. Two radically different cases occur. In the \emph{resonant case}, $\lambda_1/\lambda_2\in\mathbb Q$, the classification has been studied from the topological \cite{Camacho1978,camachosadtop}, formal and analytic \cite{MartinetRamis}, and dynamical \cite{MatteiMoussu} points of view. In the \emph{nonresonant case}, $\lambda_1/\lambda_2\notin\mathbb Q$, the holonomy is either analytically linearizable, and then the vector field is orbitally linearizable, or it is a Cremer germ \cite{Cr1}; these have been studied extensively \cite{perezmarco,perezmarcoyoccoz,P6,PerezMarco}.

Our aim is a geometric approach to the Siegel case, through the geometry of the tangency set with small spheres and the dynamics of the flow it induces on them. We work with the foliation $\F$ defined by $X$, so that classification of $\F$ means orbital classification of $X$. In this paper we focus on the resonant case.

A germ of holomorphic foliation $\F$ at $0\in\C^2$ with a \emph{resonant saddle} is generated by a vector field with linear part $\lambda z\pd_z-\mu w\pd_w$, where $\lambda/\mu=p/q$ with $p,q$ coprime positive integers. 

{\em The purpose of this paper is to show how the resonant structure of such a saddle, in particular its first formal invariant and the parabolic dynamics of its holonomy, can be read off from the way the foliation sits with respect to small round spheres around the singular point.}

The euclidean spheres we consider depend on the chart. Given local coordinates $(z,w)$ we write $S^3_{(z,w)}(r)=\{|z|^2+|w|^2=r^2\}$.  We shall 
say that $\F$ is \emph{linear Siegel in the chart} if it is generated by $h\cdot(\lambda z\pd_z-\mu w\pd_w)$ with $\lambda,\mu>0$ and $h$ a unit.

Our approach goes as follows. 
Two objects on a small sphere $S^3(r)$ are relevant. For a real hypersurface $H\subset\C^2$ and a regular point $x\in H$ of $\F$, we say that $\F$ is \emph{complex tangent} to $H$ at $x$ if $T_x\F\subset T_x^{\C}H$, and we write $M(\F,H)$ for the set of such points. If $\rho$ defines $H$ and $X$ generates $\F$, then $M(\F,H)$ is the zero set on $H$ of $X(\rho)$, so one expects it to be a curve. Off $M(\F,S^3(r))$ the foliation is transverse to the sphere, and the real lines $T\F\cap TS^3(r)$ define the \emph{characteristic flow} of $\F$ on $S^3(r)$, whose orbits are the components of the intersections of the leaves with the sphere. It is canonically oriented (Section~\ref{sec:characteristic}). As recalled above, for saddles the tangency loci are never empty; for the linear model they are Clifford tori (see the beginning of Section~\ref{sec:reduction}).

\subsection{Tangency loci and the resonance index}

Let $\F$ be a resonant Siegel saddle. A chart is \emph{admissible} if the linear part of a generator $X$ is diagonal there; after multiplying by a constant it is $X_0=\lambda z\pd_z-\mu w\pd_w$ with $\lambda/\mu=p/q$, $\gcd(p,q)=1$. We put $u=z^qw^p$, a first integral of $X_0$, and
\[
 m(\F;z,w)=\operatorname{ord}_0(X_0\wedge X)-1\in\{2,3,\dots\}\cup\{\infty\},
\]
which measures the order of contact of $\F$ with its linear part in the chart. Here $X_0\wedge X=\lambda zQ+\mu wP$ for $X=P\pd_z+Q\pd_w$; it vanishes exactly on the tangency curve between $\F$ and the linear foliation. When $m$ is finite, {\it $\Delta$ denotes the leading homogeneous part of $X_0\wedge X$}, a binary form of degree $m+1$ independent of the choice of generator with the prescribed linear part $X_0$. The number $m(\F;z,w)$ depends on the chart (Example~\ref{ex:chart}); the \emph{maximal contact order} $m(\F)$ is its supremum over admissible charts. The Clifford torus $T(1)\subset S^3(1)$ is the tangency locus of $X_0$, and the \emph{linear characteristic flow} $\lambda\pd_{\theta_1}-\mu\pd_{\theta_2}$ (with $z=|z|e^{i\theta_1}$, $w=|w|e^{i\theta_2}$) has closed orbits of class $(p,-q)$ on it.

To measure how far $\F$ is from having a first integral, we use the intrinsic quantity
\begin{equation}\label{eq:tau}
 \tau(\F)=\sup\bigl\{\operatorname{ord}_0X(\varphi):\ \varphi\in\mathcal O_{\C^2,0},\ \operatorname{ord}_0\varphi=p+q\bigr\}\in\mathbb N\cup\{\infty\},
\end{equation}
where $X$ is any generator of $\F$. Changing $X$ by a unit does not change $\operatorname{ord}_0X(\varphi)$, and $\tau(\F)$ is invariant under biholomorphisms. We call $\tau(\F)$ the \emph{resonance index} of $\F$. It is an order, not a residue: it is the best order of tangency to $\F$ of a function of the order $p+q$ of the resonant monomial $u$.

Finally, for $k\ge1$ and $m=k(p+q)+1$, an admissible chart is \emph{Poincar\'e--Dulac of order $N$} ($N\ge m$) if a generator can be written as
\begin{equation}\label{eq:normalized-generator}
 X=X_0+c\,u^kw\pd_w+\sum_{d\ge m+1}X_d,\qquad c\neq0,
\end{equation}
with $X_d$ homogeneous of degree $d$, and with $X_d$ a combination of the resonant fields $u^jz\pd_z$, $u^jw\pd_w$ whenever $d\le N$. Such charts exist for every $N$, by polynomial changes of coordinates (Section~\ref{sec:resonant}).

\begin{mainthm}[tangency geometry]\label{thm:H}
Let $\F$ be a resonant Siegel saddle with ratio $-p/q$.
\begin{enumerate}[label=\textup{(\alph*)}]
\item $\tau(\F)=\infty$ if and only if $\F$ is analytically linearizable. Otherwise $\tau(\F)=(k+1)(p+q)$ for an integer $k\ge1$, and the maximal contact order is
\[
 m(\F)=\tau(\F)-(p+q)+1=k(p+q)+1 .
\]
\item If $k<\infty$, there are polynomial charts attaining $m(\F)$ in which $\Delta$ is a multiple of $zwu^k$; Poincar\'e--Dulac charts of order $\ge m(\F)$ are such charts. In every such chart, the tangency locus with each sufficiently small sphere consists of exactly $2k$ real-analytic circles, isotopic near the Clifford torus to closed orbits of its linear characteristic flow.
\item In a Poincar\'e--Dulac chart of order $2m(\F)-1$, for all small $r>0$, each of these $2k$ circles is a normally hyperbolic circle of zeros of the characteristic flow on $S^3(r)$, of saddle type.
\end{enumerate}
\end{mainthm}

An admissible chart attains $m(\F)$ precisely when the coefficient of $zwu^k$ in $\Delta$ is nonzero (Section~\ref{sec:resonant}). Part~(b) describes a \emph{resonant break-up} of the Clifford torus: the torus of tangencies of the linear model, filled with closed orbits of the linear flow, breaks into exactly $2k$ circles, selected by the zeros of an average over these orbits (Remark~\ref{rem:average}). This is analogous to Poincar\'e's theorem on the break-up of resonant tori of nearly integrable Hamiltonian systems into isolated periodic orbits \cite[\S7.1.2]{AKN}. Figure~\ref{fig:torus} illustrates Theorem~\ref{thm:H} for $p=q=k=1$: in an optimal but non-normalized chart the tangency locus need not consist of $2k$ circles (left); in a normalized chart it consists of the $2k$ circles of~(b), along orbits of the linear flow (centre); and in a less normalized chart the saddle type of~(c) can fail on arcs (right). Part~(c) requires the deeper normalization: in charts that are only Poincar\'e--Dulac of lower order, the tangency circles may contain arcs of centre type, surrounded by closed characteristic orbits (Example~\ref{ex:elliptic}).

\begin{figure}[ht]
\centering
\includegraphics[width=\textwidth]{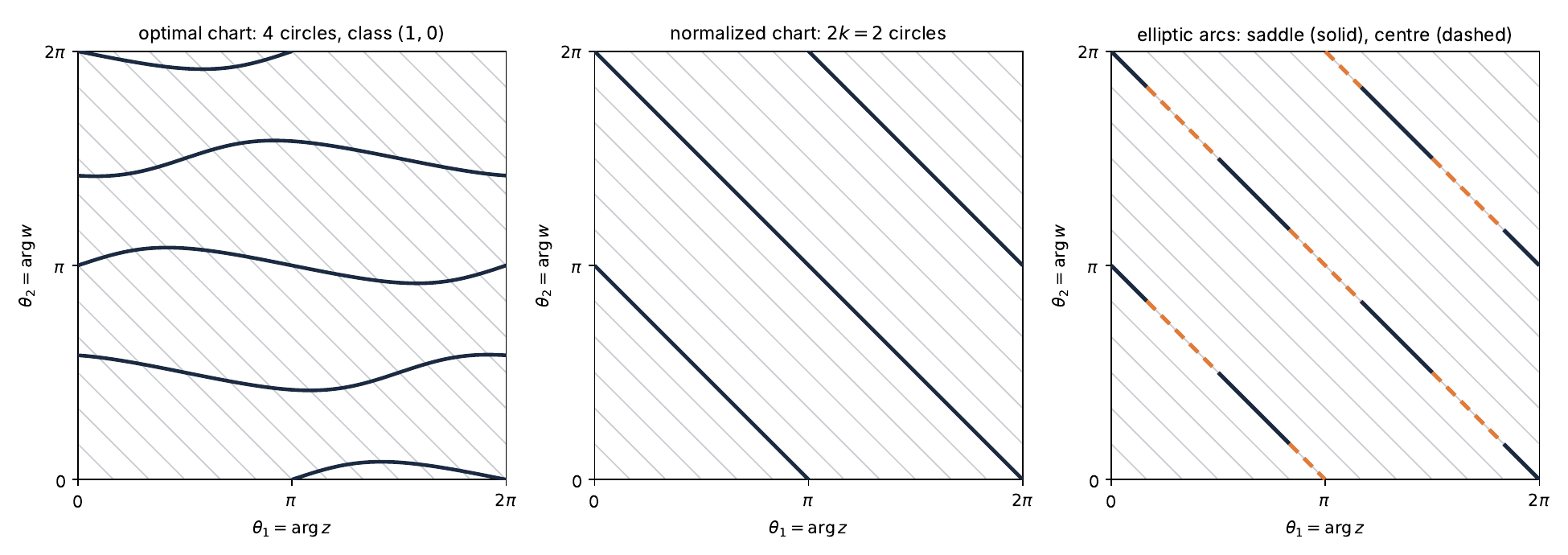}
\caption{Limits of the rescaled tangency loci $r^{-1}M(\F,S^3(r))$ on the Clifford torus $T(1)$, in the angles $(\theta_1,\theta_2)$, for $p=q=k=1$; the grey lines are the closed orbits of the linear characteristic flow. Left: the optimal but non-normalized chart of Example~\ref{ex:optimal-not-normalized}, with four circles of class $(1,0)$. Centre: a normalized chart of the same foliation, with $2k=2$ circles along orbits of the linear flow. Right: the chart of Example~\ref{ex:elliptic}, with the same two circles, solid where the contacts are of saddle type and dashed where they are of centre type.}
\label{fig:torus}
\end{figure}

\begin{remark}\label{rem:classical}
The invariant $k$ is the first formal invariant of Martinet and Ramis \cite{MartinetRamis}, and the equivalence with linearizability in~(a) is the theorem of Mattei and Moussu \cite{MatteiMoussu}. Theorem~\ref{thm:H} identifies $k$ geometrically, through contact orders and the tangency loci with small spheres.
\end{remark}

\subsection{The characteristic flow near the separatrices}

Let $S$ be the separatrix tangent to the eigendirection of the eigenvalue $\lambda$ (the $z$-axis in an admissible chart), and let $h$ be its holonomy along a positively oriented loop. It is classical, and follows from the formal classification of Martinet and Ramis \cite{MartinetRamis}, that $h'(0)=e^{-2\pi iq/p}$, that $h^p=\mathrm{id}$ if $k=\infty$, and that otherwise $h^p-\mathrm{id}$ has order exactly $kp+1$, so that $h^p$ has a Leau--Fatou flower with $kp$ attracting and $kp$ repelling petals; the same holds for the other separatrix, with $q$ in place of $p$. We give an elementary proof (Corollary~\ref{cor:holonomy}).

\begin{mainthm}[characteristic dynamics]\label{thm:char}
Let $\F$ be a resonant Siegel saddle with ratio $-p/q$ and invariant $k$, in any
analytic chart. For all small $r>0$, $\F$ is transverse to $S^3(r)$ near
$C_r=S\cap S^3(r)$, and $C_r$ is a closed orbit of the characteristic flow. Its
first-return map $P_r$, on a small transversal disc $\Sigma_r$ oriented by the sphere
and the flow, is conjugate to the holonomy $h$ of $S$ by an orientation-preserving
real-analytic diffeomorphism germ. Consequently:
\begin{enumerate}[label=\textup{(\alph*)}]
\item if $k=\infty$, every characteristic orbit near $C_r$ is closed and winds $p$
times along $C_r$;
\item if $k<\infty$, the fixed point $C_r\cap\Sigma_r$ of $P_r^p$ has Lefschetz index
$kp+1$ and a Leau--Fatou flower with $kp$ attracting and $kp$ repelling petals, whose
orbits accumulate on $C_r$ in forward, resp.\ backward, time.
\end{enumerate}
The same holds at the other separatrix, with $q$ in place of $p$. In particular $k$,
and hence linearizability, is determined by the characteristic flow near one
separatrix circle on one small sphere.
\end{mainthm}

In a normalized chart as in Theorem~\ref{thm:H}(b), the $2k$ tangency circles have winding numbers $\pm(p,-q)$, and the return map near $C_r$ has $2kp$ petals: the number of circles is $2/p$ times the number of petals, and both are governed by the same invariant. The petal structure is the chart-independent counterpart of the circle count.

Theorem~\ref{thm:char} extracts from the characteristic flow everything that a topological, or even quasiconformal, conjugacy can see. Indeed, germs tangent to the identity with the same multiplicity are topologically conjugate \cite{Camacho1978,Shcherbakov}, and topologically conjugate parabolic germs are $(1+\varepsilon)$-quasiconformally conjugate on small neighbourhoods for every $\varepsilon>0$ \cite{Jiang}. As soon as the equivalence is differentiable at the separatrix circle, the situation changes completely. By a theorem of Martinet and Ramis \cite[Ch.~III, Th\'eor\`emes~5.1--5.2]{MartinetRamis}, a $C^1$ conjugacy between non-periodic resonant holomorphic germs is biholomorphic; periodic germs are linearizable, and are classified by their multiplier. For arbitrary holomorphic germs, an elementary complexification argument (Lemma~\ref{lem:ra-conj}) gives the same conclusion for real-analytic conjugacies. The conjugacy of $P_r$ with $h$ in Theorem~\ref{thm:char} holds for every Siegel saddle (Remark~\ref{rem:irrational}), and conjugate holonomies give analytically equivalent saddles \cite[Ch.~II, Proposition~2.1]{MartinetRamis}. Together these give the following.

\begin{maincor}[rigidity of the characteristic flow]\label{cor:rigid}
Let $\F$ and $\F'$ have Siegel saddles at $0\in\C^2$, with generators having
eigenvalues $\lambda,-\mu$ and $\lambda',-\mu'$ $(\lambda,\mu,\lambda',\mu'>0)$ and
$\lambda/\mu=\lambda'/\mu'$. Let $S$ and $S'$ be the separatrices tangent to the
eigendirections of $\lambda$ and $\lambda'$. Fix analytic charts for $\F$ and $\F'$,
let $C_r=S\cap S^3(r)$ and $C'_{r'}=S'\cap S^3(r')$ for small $r,r'>0$, and suppose
that $\Phi$ is an orientation-preserving diffeomorphism from a neighbourhood of $C_r$
in $S^3(r)$ onto a neighbourhood of $C'_{r'}$ in $S^3(r')$, mapping $C_r$ to $C'_{r'}$
and oriented characteristic orbits of $\F$ to oriented characteristic orbits of $\F'$.
If
\begin{enumerate}[label=\textup{(\alph*)}]
\item $\Phi$ is real-analytic, or
\item the saddles are resonant and $\Phi$ is of class $C^1$,
\end{enumerate}
then the holonomies of $S$ and $S'$ are holomorphically conjugate, and $\F$ and $\F'$
are analytically equivalent by a germ of biholomorphism mapping $S$ to $S'$.
\end{maincor}

Thus, for a resonant saddle, the characteristic flow near a single separatrix circle on a single small sphere sees only $k$ up to topological or quasiconformal equivalence, and the whole analytic class up to $C^1$ equivalence; there is no intermediate, formal, level.

\subsection{The characteristic flow on the whole sphere}

Theorem~\ref{thm:char} describes the characteristic flow near the two separatrix circles, and Theorem~\ref{thm:H} describes the tangency circles, along which the characteristic field vanishes. In the normal-form model, the characteristic flow on a small sphere is integrable, and the two descriptions fit together completely. Let
\[
 R_\theta(z,w)=(e^{ip\theta}z,\ e^{-iq\theta}w),\qquad \theta\in\R/2\pi\Z ,
\]
and let $C_r$ and $C'_r$ be the circles in which the separatrices $\{w=0\}$ and $\{z=0\}$ meet $S^3(r)$.

\begin{mainthm}[the characteristic flow of the model]\label{thm:model}
Let $X=X_0+c\,u^kw\pd_w$ with $\lambda/\mu=p/q$, $k\ge1$, $c\ne0$, put $m=k(p+q)+1$, and let $0<r<r_0=(e\mu/|c|)^{1/(m-1)}$.
\begin{enumerate}[label=\textup{(\alph*)}]
\item The characteristic flow on $S^3(r)$ commutes with the circle action $R$, and
\[
 K=\Re\frac{1}{c\,u^k}+\frac{kp}{\lambda}\log|z|
\]
is a first integral of it on $S^3(r)\setminus(C_r\cup C'_r)$.
\item The tangency set consists of $2k$ orbits $M_1,\dots,M_{2k}$ of $R$. Each $M_j$ is a normally hyperbolic circle of zeros of the characteristic field, of saddle type. With $M_j$ removed, its stable and its unstable manifold are each the disjoint union of two cylinders, one consisting of orbits that accumulate on $C_r$ and one of orbits that accumulate on $C'_r$.
\item Every characteristic orbit not contained in $C_r\cup C'_r\cup\bigcup_j\bigl(M_j\cup W^s(M_j)\cup W^u(M_j)\bigr)$ has its $\alpha$- and $\omega$-limit sets in $\{C_r,C'_r\}$. These orbits fill an open set with exactly $6k$ components, all $R$-invariant: $2k$ consist of orbits homoclinic to $C_r$, $2k$ of orbits homoclinic to $C'_r$, and $2k$ of orbits heteroclinic between $C_r$ and $C'_r$. The nonwandering set of the characteristic flow is $C_r\cup C'_r\cup M_1\cup\dots\cup M_{2k}$. In $S^3(r)\setminus(C_r\cup C'_r)$, each component of the first two kinds is bounded by the invariant manifolds of a single $M_j$, and for each of these two kinds $j$ determines the component uniquely; in the whole sphere its boundary also contains the separatrix circle on which its orbits accumulate.
\item Let $x=(r,0)\in C_r$ and let $\Sigma=S^3(r)\cap\{z\in\R_{>0}\}$ near $x$, with coordinate $w$; it is transverse to $C_r$. If $c>0$ (which can be arranged by rotating $z$), the $2kp$ directions $\{\cos(kp\arg w)=0\}$ are the attracting and repelling axes of the Leau--Fatou flower of $P_r^p$ at $x$. Every closed subsector of each of the $2kp$ open sectors between consecutive axes consists, near $x$, of points whose orbits are homoclinic to $C_r$. The orbits running between $C_r$ and $C'_r$ meet $\Sigma$, near $x$, only in cuspidal neighbourhoods of the axes. The intersections of the $2k$ homoclinic components at $C_r$ with $\Sigma$ have $2kp$ curvilinear sector germs, asymptotic to these open sectors. They fall into $2k$ classes of $p$, indexed by the component in~(c), equivalently by $M_j$, and $P_r$ permutes the curvilinear sector germs of each class cyclically. The same holds at $C'_r$, with $q$ in place of $p$.
\end{enumerate}
\end{mainthm}

\begin{figure}[ht]
\centering
\includegraphics[width=.95\textwidth]{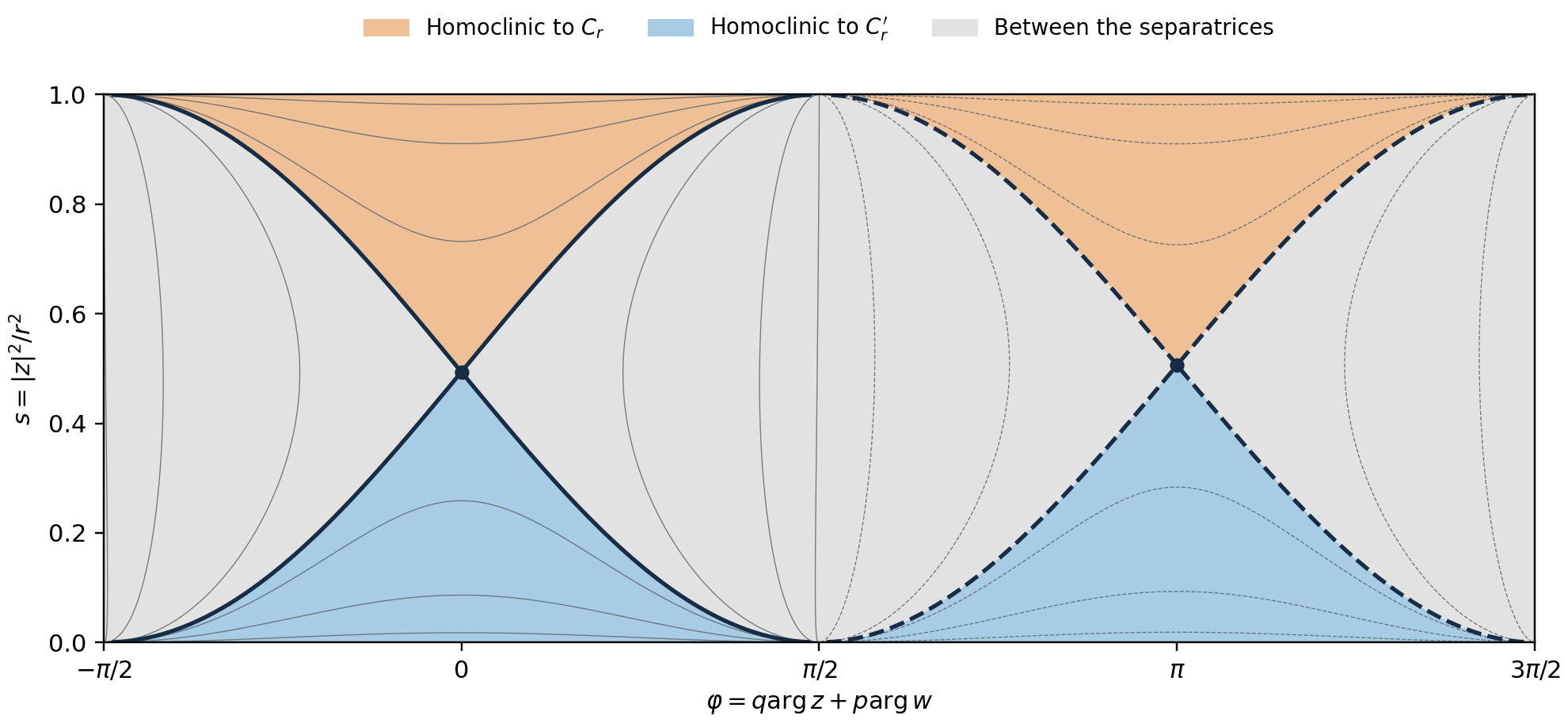}
\caption{The characteristic flow of the model, in the quotient by the circle action $R$, with coordinates $s=|z|^2/r^2$ and $\varphi=q\arg z+p\arg w$, for $p=q=k=\lambda=\mu=1$ and $c\,r^{m-1}=0.05$, over one period $-\pi/2\le\varphi\le3\pi/2$. The lines $s=1$ and $s=0$ are the separatrix circles $C_r$ and $C'_r$. The dots are the tangency circles, at $\varphi=0$ (valley sector, $\cos\varphi>0$) and $\varphi=\pi$ (ridge sector, $\cos\varphi<0$), and the thick curves are their invariant manifolds. Curves are solid in the valley sector and dashed in the ridge sector. Orange: orbits homoclinic to $C_r$. Blue: orbits homoclinic to $C'_r$. Grey: orbits heteroclinic between the two separatrix circles. Thin curves are level curves of $K$, which are orbits.}
\label{fig:portrait}
\end{figure}

Figure~\ref{fig:portrait} illustrates Theorem~\ref{thm:model} for $p=q=k=1$: the dots and thick curves are the tangency circles and their invariant cylinders of~(b), the coloured regions are the $6k=6$ components of~(c), and near $s=1$ and $s=0$ the grey regions of heteroclinic orbits reach the separatrix circles only along the axes, as in~(d). Thus the integer $k$ appears three times in the model, and Theorem~\ref{thm:model}(d) shows that these appearances are one structure: the $2k$ tangency circles, the $2kp$ inter-axis sectors at $C_r$ and the $2kq$ inter-axis sectors at $C'_r$. Each tangency circle bounds, through its invariant manifolds, one cycle of $p$ inter-axis sectors at $C_r$ and one cycle of $q$ inter-axis sectors at $C'_r$. By Theorem~\ref{thm:H}(c), the tangency circles keep the saddle structure of the model in every Poincar\'e--Dulac chart of order $2m-1$. The bound $r<r_0$ is an explicit sufficient condition for the level-set analysis in Section~\ref{sec:whole}; it is not asserted to be sharp. The characteristic flow of the model thus has the structure of a Morse--Smale flow \cite{PalisDeMelo}: a finite nonwandering set of critical elements, joined by homoclinic and heteroclinic orbits along their invariant manifolds, with no other recurrence; except that its two periodic orbits $C_r$ and $C'_r$ are not hyperbolic but parabolic, their return maps having Leau--Fatou flowers (Theorem~\ref{thm:char}).

The family of round spheres is part of the choice of local coordinates, and Theorems~\ref{thm:H}(c) and~\ref{thm:model} depend on it. Corollary~\ref{cor:rigid}, by contrast, is independent of the chart: it recovers the analytic class from the characteristic flow near a separatrix circle.

\subsection{Earlier work and connections}
For linear distributions in $\C^n$, the contact variety with spheres is the set of coneigenvectors of the coefficient matrix, and transversality was characterized in \cite{ItoScarduaYamagishi2010}. Perturbations were treated in \cite{ItoScarduaYamagishi2015}, where generic linear parts produce at least $n$ branches of the contact variety. The nondegeneracy of contacts, i.e.\ the Morse type of the distance function on the leaves, was introduced in \cite{ItoScardua2010}, and degenerate contacts were studied in \cite{ItoScarduaYamagishi2014}. Everywhere-transverse boundary dynamics were studied by Brunella \cite{Brunella1995}, alongside the classification of transversely holomorphic flows \cite{BrunellaGhys1995,Ghys1996,GeigesGonzalo2016}. Identifying a return map with holonomy is classical \cite{Camacho-LN}. Corollary~\ref{cor:rigid}(b) combines this with Martinet--Ramis rigidity \cite[Ch.~III, \S5]{MartinetRamis} (see \cite[Theorem~1.5]{Bracci}) and reconstruction from holonomy \cite[Ch.~II, Proposition~2.1]{MartinetRamis}, due to Moussu (see also \cite{ElizarovIlyashenko}). Remark~\ref{rem:classical} identifies the other classical ingredients.

Our geometric contributions are the tangency-locus descriptions (Theorems~\ref{thm:H} and~\ref{thm:C}), the transverse expansion and finite-normalization criterion (Theorem~\ref{thm:transverse}), and the whole-sphere model portrait (Theorem~\ref{thm:model}). For $n=2$, the criterion of \cite{ItoScarduaYamagishi2010} is equivalent to the condition $0\notin W(A)$ on the numerical range (Remark~\ref{rem:coneigen}); what we add is the stability of the linear picture under higher-order terms, with the exact count and linking of the branches (Proposition~\ref{thm:G}). Lemma~\ref{lem:meaning} recasts the Morse type of a contact as the transverse type of the induced characteristic flow. Theorem~\ref{thm:model} combines the classical normal-form first integral \cite{MartinetRamis} with circle symmetry.

\subsection*{Organization}
Sections~\ref{sec:prelim}--\ref{sec:leading} develop the torus reduction and prove Theorem~\ref{thm:C}, which describes the tangency loci in a fixed chart. Section~\ref{sec:numrange} treats arbitrary linear parts. Section~\ref{sec:resonant} proves Theorem~\ref{thm:H}(a),(b) and the holonomy statement, and Section~\ref{sec:characteristic} proves Theorem~\ref{thm:char} and Corollary~\ref{cor:rigid}, including the irrational case (Remark~\ref{rem:irrational}). Section~\ref{sec:whole} proves the transverse expansion, Theorem~\ref{thm:H}(c) and Theorem~\ref{thm:model}.

\section{Preliminaries}\label{sec:prelim}

Let $X=P\pd_z+Q\pd_w$ be holomorphic and let $H=\{\rho=0\}$ be a real hypersurface. Then $\F$ is complex tangent to $H$ at a regular point $x\in H$ if and only if $X(\rho)(x)=P\rho_z+Q\rho_w=0$. Multiplying $X$ by a nonvanishing function does not change $M(\F,H)$, and a biholomorphism $L$ maps $M(\F,H)$ to $M(L_*\F,L(H))$.

\begin{lemma}[normalization]\label{lem:normalization}
Let $X=X_0+\sum_{k\ge2}X_k$ be a holomorphic vector field with $X_0=\lambda z\pd_z-\mu w\pd_w$, $\lambda\mu\ne0$, and $X_k$ homogeneous of degree $k$.
\begin{enumerate}[label=\textup{(\arabic*)}]
\item A homogeneous field $X_k$ satisfies $X_0\wedge X_k=0$ if and only if $X_k=R\,X_0$ for a homogeneous polynomial $R$ of degree $k-1$.
\item $X_0\wedge X\equiv0$ if and only if $X=hX_0$ for a holomorphic function $h$ with $h(0)=1$.
\item If $X_0\wedge X\not\equiv0$ has order $m+1$ and leading part $\Delta$, there is a holomorphic unit $v$ with $v(0)=1$ such that
\[
 vX=X_0+X'_m+O(|(z,w)|^{m+1}),\qquad X_0\wedge X'_m=\Delta .
\]
\item $\Delta$ is unchanged if $X$ is replaced by $vX$ with $v$ any unit with $v(0)=1$.
\end{enumerate}
\end{lemma}

\begin{proof}
The identity $X_0\wedge X_k=\lambda zQ+\mu wP=0$ forces $z\mid P$, $w\mid Q$, and $X_k=RX_0$; the same divisibility argument for convergent series proves (2). For (3), the homogeneous parts $X_0\wedge X_k$, $k<m$, of $X_0\wedge X$ have degree $k+1<m+1$, so they vanish. Thus each term $X_k$ below degree $m$ is an orbital multiple $RX_0$ by (1), and multiplying by $1-R$ removes it. Finitely many such multiplications leave the leading wedge unchanged. Finally $X_0\wedge vX=v(X_0\wedge X)$, proving (4).
\end{proof}

\begin{lemma}[changes of coordinates tangent to the identity]\label{lem:coordchange}
Let $X=X_1+\sum_{j\ge2}X_j$ be a holomorphic vector field with linear part $X_1(\zeta)=A\zeta$, and let $\Psi(\zeta)=\zeta+V(\zeta)$, where $V=V_n+O(|\zeta|^{n+1})$ is holomorphic and $V_n$ is homogeneous of degree $n\ge2$. Then
\[
 \Psi_*X=X+[X_1,V_n]+O(|\zeta|^{n+1}),\qquad [X_1,V_n]:=DV_n\cdot X_1-DX_1\cdot V_n .
\]
In particular, $\Psi_*X$ and $X$ have the same homogeneous terms in all degrees $<n$, and the same linear part. If $V=O(|\zeta|^{n})$ with $V_n=0$, then $\Psi_*X-X=O(|\zeta|^{n+1})$.
\end{lemma}

\begin{proof}
By definition $(\Psi_*X)(\xi)=D\Psi(\zeta)\,X(\zeta)$ with $\zeta=\Psi^{-1}(\xi)$. Since $\zeta=\xi-V(\zeta)$, we have $\zeta=\xi-V(\xi)+O(|\xi|^{2n-1})$, and $2n-1\ge n+1$. Taylor expansion gives
\[
 X(\zeta)=X(\xi)-DX(\xi)V(\xi)+O(|\xi|^{2n})=X(\xi)-DX_1\,V_n(\xi)+O(|\xi|^{n+1}),
\]
because $DX_j(\xi)V(\xi)=O(|\xi|^{j-1+n})$ with $j-1+n\ge n+1$ for $j\ge2$, and $DX_1(V-V_n)=O(|\xi|^{n+1})$. Moreover $D\Psi(\zeta)X(\zeta)=X(\zeta)+DV(\zeta)X(\zeta)$ and $DV(\zeta)X(\zeta)=DV_n(\xi)X_1(\xi)+O(|\xi|^{n+1})$. Adding up gives the formula. The last assertion is the case where the degree $n$ part of $V$ vanishes.
\end{proof}

\section{Reduction to a function on the torus}\label{sec:reduction}

Put $s=|z|^2$ and $t=|w|^2$. Fix $\lambda,\mu>0$ and put
\[
 s_0=\frac{\mu}{\lambda+\mu},\qquad t_0=\frac{\lambda}{\lambda+\mu},
\]
so that the Clifford torus is $T(1)=\{s=s_0,\ t=t_0\}\subset S^3(1)$. We use the defining function $\rho=|z|^2+|w|^2-1$ of $S^3(1)$. For $X_0=\lambda z\pd_z-\mu w\pd_w$ we have, on $S^3(1)$,
\begin{equation}\label{eq:G0}
 G_0:=X_0(\rho)=\lambda s-\mu(1-s)=(\lambda+\mu)(s-s_0).
\end{equation}
This vanishes exactly on $T(1)$, and it equals $\lambda$ on $\{w=0\}$ and $-\mu$ on $\{z=0\}$. Hence $M(X_0,S^3(1))=T(1)$. We parametrize $S^3(1)\cap\{zw\neq0\}$ by $(s,\theta)\in(0,1)\times\T^2$, $\theta=(\theta_1,\theta_2)$, via
\[
 z=\sqrt s\,e^{i\theta_1},\qquad w=\sqrt{1-s}\,e^{i\theta_2}.
\]

\begin{lemma}[reduction]\label{lem:reduction}
Let $\F$ be generated near $0$ by $X=X_0+O(2)$ with $X_0\wedge X\not\equiv0$, and let $m$ and $\Delta$ be as in the Introduction. There are $r_0>0$ and real-analytic functions $\sigma,E:(-r_0,r_0)\times\T^2\to\R$ with the following properties. Put $s_r(\theta)=s_0+r^{m-1}\sigma(r,\theta)$ and
\[
 j_r(\theta)=\bigl(\sqrt{s_r(\theta)}\,e^{i\theta_1},\ \sqrt{1-s_r(\theta)}\,e^{i\theta_2}\bigr),\qquad \Sigma_r=j_r(\T^2)\subset S^3(1) .
\]
Then for $0<r<r_0$:
\begin{enumerate}[label=\textup{(\arabic*)}]
\item $j_r$ is an embedding, $\Sigma_r$ is a torus, and $j_r\to j_0$ with $j_0(\T^2)=T(1)$;
\item $M(\F,S^3(r))=r\,j_r\bigl(\{\theta:\ E(r,\theta)=0\}\bigr)$;
\item $E(0,\theta)=\kappa\,f(\theta)$, where $\kappa=1/(\lambda+\mu)$ and
\[
 f(\theta)=\Im\frac{\Delta(z,w)}{zw}\Big|_{(z,w)=j_0(\theta)}=\Im\frac{\Delta(\sqrt{s_0}e^{i\theta_1},\sqrt{t_0}e^{i\theta_2})}{\sqrt{s_0t_0}\,e^{i(\theta_1+\theta_2)}} .
\]
\end{enumerate}
\end{lemma}

\begin{proof}
By Lemma~\ref{lem:normalization}(3) we may replace $X$ by $vX$ and assume $X=X_0+X_m+O(m+1)$ with $X_0\wedge X_m=\Delta$. Write $X_m=P_m\pd_z+Q_m\pd_w$ and set $\widetilde X_r(\zeta)=r^{-1}X(r\zeta)$. Since the homothety is a linear biholomorphism, $\F$ is tangent to $S^3(r)$ at $r\zeta$ if and only if $\widetilde X_r$ is tangent to $S^3(1)$ at $\zeta$. Moreover
\[
 \widetilde X_r=X_0+r^{m-1}X_m+r^mR_r,
\]
where $R_r(\zeta)$ is holomorphic in $\zeta$ and real-analytic in $r$ near $r=0$.

\emph{Localization.} Fix $\delta>0$ small and let $U=\{|s-s_0|\le\delta\}\subset S^3(1)$. By~\eqref{eq:G0}, $|X_0(\rho)|\ge c>0$ on the compact set $S^3(1)\setminus U$, hence $\widetilde X_r(\rho)\neq0$ there for $|r|$ small. Thus all tangency points lie in $U$.

\emph{The graph $\Sigma_r$.} On $U$ put $G_r=\widetilde X_r(\rho)$. Since $\rho_z=\ol z$ and $\rho_w=\ol w$,
\begin{equation}\label{eq:Gr}
 G_r=G_0+r^{m-1}F_m+r^mR_r(\rho),\qquad F_m=\ol z\,P_m+\ol w\,Q_m,
\end{equation}
and $G_0=(\lambda+\mu)(s-s_0)$ is real. Substituting $s=s_0+r^{m-1}\sigma$ and dividing by $r^{m-1}$, the equation $\Re G_r=0$ becomes, for $r\ne0$,
\[
 \Psi(r,\theta,\sigma):=(\lambda+\mu)\,\sigma+\Re F_m(s_0+r^{m-1}\sigma,\theta)+r\,\Re R_r(\rho)(s_0+r^{m-1}\sigma,\theta)=0 .
\]
At $r=0$ this has the unique solution $\sigma_0(\theta)=-\Re F_m(s_0,\theta)/(\lambda+\mu)$, and $\pd_\sigma\Psi=\lambda+\mu>0$ there. By the implicit function theorem and compactness of $\T^2$ we obtain a real-analytic $\sigma(r,\theta)$ for $|r|<r_0$. For fixed small $r\neq0$ and $\theta$, the function $s\mapsto\Re G_r(s,\theta)$ is strictly increasing on $U$, since its derivative is $\lambda+\mu+O(r^{m-1})>0$. Hence $s_r(\theta)$ is the only zero of $\Re G_r(\cdot,\theta)$ in $U$, and $\{\Re G_r=0\}\cap U=\Sigma_r$. This proves (1).

\emph{The function $E$.} On $\Sigma_r$, $G_r=i\Im G_r$. Since $\Im G_0=0$, \eqref{eq:Gr} gives $\Im G_r=r^{m-1}\bigl(\Im F_m+r\Im R_r(\rho)\bigr)$. Define
\[
 E(r,\theta)=\Im F_m(s_r(\theta),\theta)+r\,\Im R_r(\rho)(s_r(\theta),\theta),
\]
which is real-analytic up to $r=0$. For $0<r<r_0$ the zeros of $G_r$ on $S^3(1)$ are exactly the points $j_r(\theta)$ with $E(r,\theta)=0$, which gives (2).

\emph{Computation at $r=0$.} We have $E(0,\theta)=\Im F_m(s_0,\theta)$. On $T(1)$, $\ol z=s_0/z$ and $\ol w=t_0/w$, and $s_0=\kappa\mu$, $t_0=\kappa\lambda$. Hence
\[
 F_m=\frac{s_0\,wP_m+t_0\,zQ_m}{zw}=\kappa\,\frac{\mu wP_m+\lambda zQ_m}{zw}=\kappa\,\frac{X_0\wedge X_m}{zw}=\kappa\,\frac{\Delta}{zw}\qquad\text{on }T(1).
\]
This proves (3).
\end{proof}

\section{Tangency loci in a fixed chart}\label{sec:leading}

The proof of Theorem~\ref{thm:H}(b) rests on a description of the tangency loci with small spheres in a fixed chart. It does not use resonance, and we state it for every nondegenerate germ whose eigendirections are orthogonal in the chart; Section~\ref{sec:numrange} treats the other cases. Suppose $\F$ is generated by
\[
 X=X_0+O(|(z,w)|^2),\qquad X_0=\lambda z\frac{\pd}{\pd z}-\mu w\frac{\pd}{\pd w},\qquad \lambda,\mu>0.
\]
The holomorphic function $X_0\wedge X$ (the determinant of the two vector fields) vanishes exactly on the tangency curve between $\F$ and the linear foliation defined by $X_0$. It vanishes identically if and only if $\F$ is linear Siegel in the chart (Lemma~\ref{lem:normalization}). Otherwise its order at $0$ is at least $3$, and we write
\[
 \operatorname{ord}_0(X_0\wedge X)=m+1,\quad m\ge2,\qquad \Delta=\text{leading homogeneous part of }X_0\wedge X .
\]
Thus $\Delta$ is a nonzero binary form of degree $m+1$; for generators with the prescribed linear part $X_0$, it depends only on $\F$ and the chart. Without this normalization its projective zero divisor is independent of the generator. Its zero set in $\PP^1$, a set of at most $m+1$ points, is the tangent cone of the tangency curve. Under the Hopf map $(z,w)\mapsto[z:w]$ the Clifford torus $T(r)$ maps onto the circle
\[
 \gamma_{\lambda,\mu}=\{[z:w]\in\PP^1:\ \lambda|z|^2=\mu|w|^2\},
\]
which bounds the disc $\{\lambda|z|^2<\mu|w|^2\}\ni[0:1]$.

\begin{theorem}[structure of the generic tangency locus]\label{thm:C}
Let $\F$ be generated in the chart $(z,w)$ by $X=X_0+O(2)$ as above, and assume that $\F$ is not linear in this chart. Let $m$ and $\Delta$ be as above. There is $r_0>0$ such that for every $0<r<r_0$ the following hold.
\begin{enumerate}[label=\textup{(\alph*)}]
\item $M(\F,S^3(r))$ is contained in a real-analytic torus $\Sigma_r\subset S^3(r)\setminus\{zw=0\}$, which is a graph over the Clifford torus $T(r)$ preserving both angular coordinates at distance $O(r^m)$. The set $M(\F,S^3(r))$ is nonempty and one-dimensional; it is a finite union of points and real-analytic arcs, and in particular it is never a surface.
\item Assume that $\Delta$ has no zero on $\gamma_{\lambda,\mu}$, and let $N$ be the number of zeros of $\Delta$ in the disc $\{\lambda|z|^2<\mu|w|^2\}$, counted with multiplicity. Put $g=\gcd(N-1,m-1)$, with the convention $\gcd(0,k)=k$. Then $M(\F,S^3(r))$ is a disjoint union of exactly $2g$ embedded real-analytic circles. Along each of them the winding numbers of $\arg z$ and $\arg w$ are
\[
 \pm\left(\frac{m-N}{g},\ \frac{1-N}{g}\right).
\]
As $r\to0$, the rescaled loci $r^{-1}M(\F,S^3(r))$ converge in the $C^1$ topology to the curve
\[
 \bigl\{(z,w)\in T(1):\ \Im\bigl(\Delta(z,w)/(zw)\bigr)=0\bigr\}.
\]
\end{enumerate}
\end{theorem}

The hypothesis of (b) is an open dense condition on the $(m+1)$-jet of $\F$ relative to the chart.

\subsection{The leading function}
Keep the notation of Lemma~\ref{lem:reduction}, and put $a_0=\sqrt{s_0}$, $b_0=\sqrt{t_0}$. Use on $\T^2$ the coordinates $(\eta,\theta_2)$ with $\eta=\theta_1-\theta_2$; this is an automorphism of $\T^2$ (its matrix lies in $SL(2,\Z)$). Define
\[
 q(\zeta)=\frac{\Delta(a_0\zeta,b_0)}{a_0b_0},\qquad G(\eta)=e^{-i\eta}q(e^{i\eta}).
\]
Since $\Delta$ is homogeneous of degree $m+1$, on $T(1)$ we have $\Delta(a_0e^{i\theta_1},b_0e^{i\theta_2})=e^{i(m+1)\theta_2}\Delta(a_0e^{i\eta},b_0)$, and therefore
\begin{equation}\label{eq:f-formula}
 f(\eta,\theta_2)=\Im\Bigl(e^{i(m-1)\theta_2}\,G(\eta)\Bigr).
\end{equation}
The zeros of $q$ correspond to the roots $[a_0\zeta:b_0]$ of $\Delta$ in $\PP^1$. In particular $q$ vanishes at a point of the unit circle if and only if $\Delta$ has a root on $\gamma_{\lambda,\mu}=\{\lambda|z|^2=\mu|w|^2\}$, and the number of zeros of $q$ in the unit disc is the number $N$ of roots of $\Delta$ in $\{\lambda|z|^2<\mu|w|^2\}$, counted with multiplicity. (A root at $[1:0]$ corresponds to a drop in the degree of $q$ and lies outside.)

\begin{lemma}\label{lem:leading}
\begin{enumerate}[label=\textup{(\arabic*)}]
\item $f$ is a real trigonometric polynomial, not identically zero, with zero mean; in particular it takes both signs.
\item $0$ is a regular value of $f$ if and only if $q$ has no zero on the unit circle.
\item In that case, with $g=\gcd(N-1,m-1)$, the set $f^{-1}(0)$ is a disjoint union of exactly $2g$ embedded circles, each of homology class $\pm\bigl(\frac{m-N}{g},\frac{1-N}{g}\bigr)$ with respect to the basis of $H_1(\T^2;\Z)$ given by the $\theta_1$- and $\theta_2$-circles.
\end{enumerate}
\end{lemma}

\begin{proof}
(1) By~\eqref{eq:f-formula}, $f$ is a trigonometric polynomial. If $f\equiv0$, then for each $\eta$ the number $e^{i(m-1)\theta_2}G(\eta)$ is real for all $\theta_2$; since $m-1\ge1$, this forces $G\equiv0$, hence $q\equiv0$ and $\Delta\equiv0$ by homogeneity, a contradiction. Integrating~\eqref{eq:f-formula} in $\theta_2$ gives zero, since $m-1\neq0$. A continuous nonzero function with zero mean takes both signs.

(2) If $G(\eta)\ne0$ for all $\eta$, then at a zero of $f$ the number $e^{i(m-1)\theta_2}G(\eta)$ is real and nonzero, so
\[
 \pd_{\theta_2}f=(m-1)\Re\bigl(e^{i(m-1)\theta_2}G(\eta)\bigr)\neq0 .
\]
If $G(\eta_0)=0$, then $f$ vanishes on the whole circle $\{\eta=\eta_0\}$, where $\pd_{\theta_2}f=0$ and $\pd_\eta f=\Im\bigl(e^{i(m-1)\theta_2}G'(\eta_0)\bigr)$; the latter vanishes for some $\theta_2$. That point is a critical point of $f$ on $f^{-1}(0)$.

(3) Let $d$ be the winding number of $G$ around $0$. By the argument principle $q(e^{i\eta})$ winds $N$ times, so $d=N-1$. Choose a continuous lift $A:\R\to\R$ of $\arg G$, so that $A(\eta+2\pi)=A(\eta)+2\pi d$. By~\eqref{eq:f-formula},
\[
 f^{-1}(0)=\bigcup_{k\in\Z}\Bigl\{\theta_2=\tfrac{k\pi-A(\eta)}{m-1}\ \bmod 2\pi\Bigr\}.
\]
Denote by $c_k$ the curve $\eta\mapsto\bigl(\eta,\frac{k\pi-A(\eta)}{m-1}\bigr)$. Then $c_{k+2(m-1)}=c_k$ and $c_k(\eta+2\pi)=c_{k-2d}(\eta)$. Distinct indices modulo $2(m-1)$ give disjoint graphs over $\eta$. Thus the components of $f^{-1}(0)$ correspond to the orbits of $k\mapsto k-2d$ on $\Z/2(m-1)\Z$. There are $\gcd(2d,2(m-1))=2g$ orbits, each of length $L=(m-1)/g$. Along the closed curve obtained by following one orbit, $\eta$ winds $L$ times and $\theta_2$ changes by $-2\pi dL/(m-1)$, i.e.\ winds $-d/g$ times. Since $\theta_1=\eta+\theta_2$, the class is $\bigl(L-\frac dg,-\frac dg\bigr)=\bigl(\frac{m-N}{g},\frac{1-N}{g}\bigr)$, which is primitive because $\gcd(m-1-d,d)=g$. The sign depends on orientation.
\end{proof}

\subsection{Proof of Theorem~\ref{thm:C}}

Apply Lemma~\ref{lem:reduction}, and let $f$ be as in Lemma~\ref{lem:leading}.

\emph{Part (a).} By Lemma~\ref{lem:reduction}, $M(\F,S^3(r))\subset r\Sigma_r$, and $r\Sigma_r$ is a graph over $T(r)$ preserving both angular coordinates at distance $r\cdot O(r^{m-1})$; it avoids the axes. Choose $\theta_\pm$ with $\pm f(\theta_\pm)>0$. Since $E(r,\cdot)\to\kappa f$ uniformly, for small $r$ the function $E(r,\cdot)$ takes both signs on the connected torus. Hence its zero set is nonempty and separates $\T^2$, and in particular it is not finite. Since $E(r,\cdot)$ is real-analytic and not identically zero, its zero set is a real-analytic subset of $\T^2$ with empty interior. Hence it is of dimension at most one and is a finite union of points and real-analytic arcs \cite{BierstoneMilman}. Since it separates $\T^2$, it has dimension exactly one. Its image under the real-analytic embedding $rj_r$ is $M(\F,S^3(r))$.

\emph{Part (b).} Assume $q$ has no zero on the unit circle. By Lemma~\ref{lem:leading}(2) and its proof, $\pd_{\theta_2}f\ne0$ on $f^{-1}(0)$, hence on a neighbourhood $V$ of $f^{-1}(0)$. Since $E\to\kappa f$ in $C^1$, for $|r|<r_1$ we have $E(r,\cdot)^{-1}(0)\subset V$ and $\pd_{\theta_2}E\neq0$ there. Thus
\[
 \mathcal Z=\{(r,\theta):\ |r|<r_1,\ E(r,\theta)=0\}
\]
is a real-analytic surface on which the projection $(r,\theta)\mapsto r$ is a proper submersion. By Ehresmann's theorem the fibres $E(r,\cdot)^{-1}(0)$ are isotopic in $\T^2$ to $E(0,\cdot)^{-1}(0)=f^{-1}(0)$, and converge to it in $C^1$. The map $j_r$ preserves the angular coordinates, so the winding numbers of $\arg z,\arg w$ along the components of $M(\F,S^3(r))=rj_r(E(r,\cdot)^{-1}(0))$ are the homology classes computed in Lemma~\ref{lem:leading}(3). Finally $r^{-1}M(\F,S^3(r))=j_r(E(r,\cdot)^{-1}(0))\to j_0(f^{-1}(0))=\{(z,w)\in T(1):\Im(\Delta/(zw))=0\}$.\qed

\subsection{Examples}

\begin{exe}\label{ex:chart}
Let $X=\lambda z\pd_z+(-\mu w+cz^2)\pd_w$ with $c\ne0$. Then $X_0\wedge X=\lambda cz^3$, so $m=2$ and $\Delta=\lambda cz^3$, whose only root $[0:1]$ has multiplicity three and lies inside $\gamma_{\lambda,\mu}$. Hence $N=3$, $g=\gcd(2,1)=1$, and for small $r$ the tangency locus consists of two circles along which $(\arg z,\arg w)$ winds $\pm(1,2)$ times. Directly, $\Im(\Delta/(zw))$ is a positive multiple of $\sin(2\theta_1-\theta_2+\arg c)$ on $T(1)$, where $z=|z|e^{i\theta_1}$, $w=|w|e^{i\theta_2}$. On the other hand, since $\lambda,\mu>0$, the polynomial change $w'=w-\frac{c}{\mu+2\lambda}z^2$ transforms $X$ into $X_0$. This illustrates that the tangency loci detect linearity \emph{in the chart}, not linearizability.
\end{exe}

\begin{exe}[optimal charts need not be normalized]\label{ex:optimal-not-normalized}
For $X=z\pd_z+(-w+zw^2+2w^3)\pd_w$, one has $p=q=1$, $m=3$ and $\Delta=zw^2(z+2w)$. Its nonzero resonant coefficient gives $\tau=4$ by Proposition~\ref{prop:tau}(3). Since $\tau$ does not depend on the chart, Proposition~\ref{prop:tau}(1) gives $m\le\tau-p-q+1=3$ in every admissible chart, so this chart is optimal and $k=1$. The roots $[0:1],[-2:1],[1:0]$ have multiplicities $1,1,2$; only $[0:1]$ lies inside $|z|<|w|$, and none lies on its boundary. Thus $N=1$ and $g=2$, giving four tangency circles of class $\pm(1,0)$ by Theorem~\ref{thm:C}(b). The polynomial chart $(z,w)\mapsto(z,w+w^3)$ removes the nonresonant cubic term, leaving leading orbital form $z^2w^2$. Theorem~\ref{thm:H}(b) then gives two circles of class $\pm(1,-1)$. Optimality alone does not give the count $2k$.
\end{exe}

In particular, by Theorem~\ref{thm:C}(a) and Lemma~\ref{lem:normalization}(2), in a chart where the linear part is $X_0$ the tangency locus with a small sphere is a surface only if $\F$ is linear in the chart, and then it is the Clifford torus.

\section{Arbitrary linear parts: the numerical range}\label{sec:numrange}

By the elliptical range theorem \cite{Li}, $W(A)$ is a closed elliptical disc with foci at the eigenvalues of $A$, which degenerates to the segment joining them exactly when $A$ is normal. Let $\pi:S^3(1)\to\PP^1\cong S^2$ be the Hopf map.

\begin{proposition}[tangency in arbitrary charts]\label{thm:G}
Let $\F$ have a nondegenerate singularity at $0$, and let the linear part of a generator be $A\zeta$ in local coordinates. Exactly one of the following holds.
\begin{enumerate}[label=\textup{(\arabic*)}]
\item $0\notin W(A)$, and $M(\F,S^3(r))=\emptyset$ for all small $r>0$.
\item $A$ is not normal and $0\in\operatorname{int}W(A)$. Then $M(A\zeta,S^3(1))$ consists of two Hopf fibres $c_\pm$, and for all small $r>0$, $M(\F,S^3(r))$ consists of two real-analytic circles, isotopic in $S^3(r)$ to $r(c_+\cup c_-)$, hence forming a Hopf link; moreover $r^{-1}M(\F,S^3(r))\to c_+\cup c_-$ in $C^1$.
\item $A$ is normal and $0\in W(A)$, i.e.\ $A$ is unitarily diagonalizable with eigenvalue ratio in $\R_{<0}$; after a unitary change, Theorem~\ref{thm:C} applies.
\item $A$ is not normal and $0\in\pd W(A)$. Then $M(A\zeta,S^3(1))$ is a single Hopf fibre $c_0$, and $r^{-1}M(\F,S^3(r))$ lies in any given neighbourhood of $c_0$ for small $r$.
\end{enumerate}
\end{proposition}

Case~(2) also occurs in the Poincar\'e domain: for $A=\bigl(\begin{smallmatrix}1&\gamma\\0&1\end{smallmatrix}\bigr)$ with $|\gamma|>2$, $W(A)$ is the disc of radius $|\gamma|/2$ about $1$.

\begin{lemma}\label{lem:bloch}
Write $\zeta\zeta^*=\frac12(I+x\cdot\sigma)$ for $\zeta\in S^3(1)$, with $\sigma=(\sigma_1,\sigma_2,\sigma_3)$ the Pauli matrices, so that $\zeta\mapsto x\in S^2$ is $\pi$. Then $\ip{A\zeta}{\zeta}=\phi(\pi(\zeta))$ with $\phi(x)=\frac12\operatorname{tr}A+L(x)$, $L(x)=\frac12\sum_ix_i\operatorname{tr}(A\sigma_i)$ real linear, and $\operatorname{rank}_\R L=2$ if and only if $A$ is not normal.
\end{lemma}

\begin{proof}
$\ip{A\zeta}{\zeta}=\operatorname{tr}(A\zeta\zeta^*)$. Write $A=H+iK$ with $H,K$ Hermitian with traceless parts $h\cdot\sigma$, $k\cdot\sigma$. Then $L(x)=h\cdot x+i\,k\cdot x$, and since $[h\cdot\sigma,k\cdot\sigma]=2i(h\times k)\cdot\sigma$, $h,k$ are dependent if and only if $H,K$ commute, i.e.\ $A$ is normal.
\end{proof}

\begin{proof}[Proof of Proposition~\ref{thm:G}]
With $X_r(\zeta)=r^{-1}X(r\zeta)$ and $g_r=\ip{X_r}{\zeta}$ on $S^3(1)$, $g_r\to g_0=\phi\circ\pi$ analytically and $M(\F,S^3(r))=r\,g_r^{-1}(0)$. If $0\notin W(A)$, $|g_0|$ is bounded below, proving (1). For nonnormal $A$, $\phi$ is an affine surjection $\R^3\to\R^2$ restricted to $S^2$, with fibre over $0$ a line $\ell$. If $0\in\operatorname{int}W(A)$, $\ell$ meets $S^2$ transversally in two points, so $g_0^{-1}(0)$ is two Hopf fibres on which $dg_0$ has rank two, and the implicit function theorem gives (2). If $\ell$ is tangent to $S^2$, the zero set of $g_0$ is one Hopf fibre and uniform convergence gives (4). If $A$ is normal, $W(A)$ is the segment joining its eigenvalues, which proves (3).
\end{proof}

\begin{remark}[coneigenvectors]\label{rem:coneigen}
Write $\omega=\iota_X(dz\wedge dw)=-Q\,dz+P\,dw$; for $X=A\zeta$ its coefficient vector is $B\zeta$, $B=\bigl(\begin{smallmatrix}0&-1\\1&0\end{smallmatrix}\bigr)A$. Since $\ip X\zeta=0$ exactly when $(-Q,P)\in\C\ol\zeta$, $M(A\zeta,S^3(1))$ is the set of unit coneigenvectors of $B$, and Proposition~\ref{thm:G}(1) is the case $n=2$ of the criterion of \cite{ItoScarduaYamagishi2010}: transversality holds if and only if $\ol BB$ has no positive eigenvalue. More precisely, cases (1)--(4) correspond to: no positive eigenvalue of $\ol BB$; two distinct ones; $\ol BB=\sigma^2I$; a double one with one-dimensional eigenspace. Indeed, $S(\zeta)=\ol B\,\ol\zeta$ is antilinear with $S^2=\ol BB$, coneigenvectors span the $S$-invariant lines, and these lie in positive eigenspaces $E$ of $\ol BB$, which $S$ preserves. A line $E$ gives one Hopf fibre; if $E=\C^2$, $S/\sigma$ is a real structure whose invariant lines give a torus of Hopf fibres. Since $\ol BB$ is similar to its conjugate and $\det\ol BB=|\det B|^2>0$, these are the only possibilities. In case~(2), \cite[Corollary~3.7]{ItoScarduaYamagishi2015} gives at least two branches of the contact variety for every perturbation; Proposition~\ref{thm:G}(2) shows that in $\C^2$ there are exactly two, forming a Hopf link on each small sphere.
\end{remark}

\section{The resonant case}\label{sec:resonant}

Throughout this section $\F$ is a resonant Siegel saddle, the chart $(z,w)$ is admissible, and the generator $X$ has linear part $X_0=\lambda z\pd_z-\mu w\pd_w$ with $\lambda/\mu=p/q$, $\gcd(p,q)=1$. We put $u=z^qw^p$. A monomial $z^iw^j$ is \emph{resonant} if $(i-1)\lambda=(j-1)\mu$, i.e.\ if $z^iw^j=zw\,u^l$ for some $l\ge0$. In each degree there is at most one resonant monomial, and there is one in degree $d$ exactly when $d=l(p+q)+2$.

\subsection{Homological equations}

\begin{lemma}\label{lem:homological}
\begin{enumerate}[label=\textup{(\arabic*)}]
\item The operator $V\mapsto[X_0,V]=DV\cdot X_0-DX_0\cdot V$ acts diagonally on monomial vector fields:
\[
 [X_0,z^aw^b\pd_z]=\bigl((a-1)\lambda-b\mu\bigr)z^aw^b\pd_z,\qquad
 [X_0,z^aw^b\pd_w]=\bigl(a\lambda-(b-1)\mu\bigr)z^aw^b\pd_w .
\]
Its kernel on homogeneous fields of degree $n$ is spanned by $u^jz\pd_z$ and $u^jw\pd_w$ when $n=j(p+q)+1$, and is zero otherwise.
\item For every homogeneous $V$ of degree $n\ge2$, $X_0\wedge[X_0,V]$ is a homogeneous polynomial of degree $n+1$ without resonant monomials. Conversely, every homogeneous polynomial of degree $n+1$ without resonant monomials is of this form.
\end{enumerate}
\end{lemma}

\begin{proof}
(1) is a direct computation: $DV\cdot X_0$ multiplies $z^aw^b$ by $a\lambda-b\mu$, and $DX_0\cdot V$ multiplies the $\pd_z$ (resp.\ $\pd_w$) component by $\lambda$ (resp.\ $-\mu$). The eigenvalue vanishes exactly for $z^{jq+1}w^{jp}\pd_z=u^jz\pd_z$ and $z^{jq}w^{jp+1}\pd_w=u^jw\pd_w$.

(2) We have $X_0\wedge z^aw^b\pd_z=\mu z^aw^{b+1}$ and $X_0\wedge z^aw^b\pd_w=\lambda z^{a+1}w^b$. By (1), the monomial $z^iw^j$ produced from $V=z^aw^b\pd_z$ (where $i=a$, $j=b+1$) or from $V=z^aw^b\pd_w$ (where $i=a+1$, $j=b$) appears with coefficient a nonzero multiple of $(i-1)\lambda-(j-1)\mu$. This vanishes exactly in the resonant case, and every monomial $z^iw^j$ with $i+j=n+1$ arises in at least one of the two ways.
\end{proof}

\subsection{The resonance index}

\begin{proposition}\label{prop:tau}
Let $m=m(\F;z,w)$ be finite.
\begin{enumerate}[label=\textup{(\arabic*)}]
\item $\tau(\F)\ge m+p+q-1$.
\item If the coefficient of $\Delta$ on its resonant monomial vanishes (in particular, if $\Delta$ has no resonant monomial), then there is $V$ homogeneous of degree $m$ such that $\zeta\mapsto\zeta+V(\zeta)$ produces an admissible chart with $m(\F)>m$.
\item If this coefficient is nonzero, then $m=l(p+q)+1$ for some integer $l\ge1$, and $\tau(\F)=(l+1)(p+q)=m+p+q-1$. Moreover, after a polynomial change $\zeta\mapsto\zeta+V(\zeta)$ with $V$ homogeneous of degree $m$ and multiplication by a unit, $X=X_0+c\,u^lw\pd_w+O(|\zeta|^{m+1})$, where $\lambda c$ is the given coefficient.
\end{enumerate}
\end{proposition}

\begin{proof}
By Lemma~\ref{lem:normalization}(3) we may assume $X=X_0+X_m+O(|\zeta|^{m+1})$ with $X_0\wedge X_m=\Delta$.

(1) We have $X(u)=X_m(u)+O(|\zeta|^{m+p+q})$, since $X_0(u)=0$. Hence $\operatorname{ord}_0X(u)\ge m+p+q-1$.

(2) By Lemma~\ref{lem:homological}(2), $\Delta=-X_0\wedge[X_0,V]$ for some $V$ of degree $m$. By Lemma~\ref{lem:coordchange}, in the new chart the generator is $X_0+(X_m+[X_0,V])+O(|\zeta|^{m+1})$, and $X_0\wedge(X_m+[X_0,V])=0$.

(3) The resonant monomial has degree $m+1=l(p+q)+2$, and $l\ge1$ because $m\ge2$. Write $\Delta=\lambda c\,zwu^l+\Delta_{\mathrm{nr}}$ and remove $\Delta_{\mathrm{nr}}$ as in (2). The new $X_m$ satisfies $X_0\wedge(X_m-c\,u^lw\pd_w)=0$, because $X_0\wedge c\,u^lw\pd_w=\lambda c\,zwu^l$. By Lemma~\ref{lem:normalization}(1), $X_m=c\,u^lw\pd_w+RX_0$, and multiplying by the unit $1-R$ gives $X=X_0+c\,u^lw\pd_w+O(|\zeta|^{m+1})$.

Since $\tau$ is invariant, we may compute it in this chart. First, $X(u)=c\,u^lw\pd_w(u)+O(|\zeta|^{m+p+q})=cp\,u^{l+1}+O(|\zeta|^{(l+1)(p+q)+1})$, so $\tau\ge(l+1)(p+q)$. Conversely, let $\operatorname{ord}_0U=p+q$ with homogeneous expansion $U=\sum_dU_d$, and suppose $\operatorname{ord}_0X(\varphi)\ge(l+1)(p+q)$. The remainder $O(|\zeta|^{m+1})$ applied to $U$ has order at least $m+p+q=(l+1)(p+q)+1$. The term $c\,u^lw\pd_w$ applied to $U_e$ has degree $e+l(p+q)\ge(l+1)(p+q)$. Hence for $p+q\le d<(l+1)(p+q)$ the degree $d$ part of $X(\varphi)$ is $X_0(U_d)$, which must vanish. In degree $p+q$, the kernel of $X_0$ is spanned by $u$, so $U_{p+q}=\alpha u$ with $\alpha\ne0$. The degree $(l+1)(p+q)$ part of $X(\varphi)$ is
\[
 X_0(U_{(l+1)(p+q)})+c\alpha p\,u^{l+1}.
\]
Its first term contains no resonant monomial, since $X_0$ multiplies $z^aw^b$ by $a\lambda-b\mu$. Its second term is a nonzero multiple of the resonant monomial $u^{l+1}$. Hence it is nonzero, and $\operatorname{ord}_0X(\varphi)=(l+1)(p+q)$.
\end{proof}

\subsection{Proof of Theorem~\ref{thm:H}(a) and (b)}

(a) If $\tau(\F)<\infty$, no admissible chart can have $m=\infty$, since the pullback of $u$ from such a chart would be an exact first integral of order $p+q$. Start from any admissible chart and apply Proposition~\ref{prop:tau}(2) as long as the resonant coefficient of $\Delta$ vanishes. Each step increases $m$, and $m\le\tau-p-q+1$ by Proposition~\ref{prop:tau}(1). So after finitely many polynomial changes we reach a chart where the coefficient is nonzero. There, by Proposition~\ref{prop:tau}(3), $\tau=(k+1)(p+q)$ for some integer $k\ge1$, and $m=k(p+q)+1=\tau-p-q+1$, the maximal value. By Proposition~\ref{prop:tau}(2) and~(3), a chart attains this value if and only if its resonant coefficient is nonzero.

If $\tau(\F)=\infty$, Proposition~\ref{prop:tau}(3) shows that, in any admissible chart with finite $m$, the leading orbital term has zero resonant coefficient. Apply Proposition~\ref{prop:tau}(2) whenever $m$ is finite. If a finite stage gives $X_0\wedge X\equiv0$, then Lemma~\ref{lem:normalization}(2) gives exact orbital linearity in that chart, and $u$ pulls back to a nonconstant holomorphic first integral. Otherwise the procedure produces changes $\zeta\mapsto\zeta+V_j(\zeta)$ with $\deg V_j$ strictly increasing, and $m$ is unbounded. Their composition converges in the formal topology to a formal change $\widehat\Phi$ tangent to the identity. The unit multiplications used in Lemma~\ref{lem:normalization}(3) also converge formally, since their orders tend to infinity. In the resulting formal coordinates the transformed generator is $\widehat h X_0$ with $\widehat h$ a formal unit. Thus $u\circ\widehat\Phi$ is a nonconstant formal first integral of $\F$. By Mattei--Moussu \cite{MatteiMoussu1978,MatteiMoussu}, since the singularity is isolated, $\F$ then has a nonconstant holomorphic first integral. In either case a saddle with a holomorphic first integral is analytically linearizable \cite{MatteiMoussu}, and the supremum of $m$ is infinite. Conversely, if $\F$ is linearizable, let $\Psi$ be a linearizing chart. Then $m=\infty$ in this chart, and $U=(z^qw^p)\circ\Psi$ has order $p+q$ and satisfies $X(\varphi)=0$. Hence $\tau(\F)=\infty$.

(b) Proposition~\ref{prop:tau}(3) gives an optimal polynomial chart in which $X=X_0+c\,u^kw\pd_w+O(|\zeta|^{m+1})$, so that $\Delta=\lambda c\,z^{kq+1}w^{kp+1}$. Its roots in $\PP^1$ are $[0:1]$, with multiplicity $kq+1$, inside $\{\lambda|z|^2<\mu|w|^2\}$, and $[1:0]$, outside; neither lies on $\gamma_{\lambda,\mu}$. Theorem~\ref{thm:C}(b) applies with $N=kq+1$, $m-1=k(p+q)$ and $g=\gcd(kq,k(p+q))=k$. It gives $2k$ circles with winding numbers $\pm\bigl(\frac{k(p+q)-kq}{k},\frac{-kq}{k}\bigr)=\pm(p,-q)$, the class of the closed orbits of the linear characteristic flow. The limit curve $\{\Im(\lambda c\,u^k)=0\}\subset T(1)$ is a union of $2k$ such orbits. A Poincar\'e--Dulac chart of order $N\ge m$ has $X_m=c\,u^kw\pd_w$, hence $\Delta=\lambda c\,zwu^k$; such charts exist for every $N$, by Proposition~\ref{prop:tau}(3) followed by the Poincar\'e--Dulac theorem \cite[Theorem~4.10]{Ilyashenko-Yakovenko} in degrees $>m$.\qed

\begin{remark}[the averaging mechanism]\label{rem:average}
In any optimal chart, normalized or not, $k$ is visible on the Clifford torus through an average. Let $f=\Im(\Delta/(zw))|_{T(1)}$, whose zero set is the limit of the rescaled tangency loci (Theorem~\ref{thm:C}(b)). The average of $f$ over each closed orbit of the linear characteristic flow vanishes on exactly $2k$ orbits, although $f^{-1}(0)$ itself need not consist of $2k$ circles (Example~\ref{ex:optimal-not-normalized}); in a normalized chart, $f^{-1}(0)$ is the union of these $2k$ orbits. Indeed, in an optimal chart the resonant coefficient $\lambda c$ is nonzero, by the proof of Theorem~\ref{thm:H}(a). On $T(1)$, $\Delta/(zw)=\lambda c\,u^k+\sum_na_ne^{in\cdot\theta}$. The frequencies $n$ occurring in the sum satisfy $n_1+n_2=m-1=k(p+q)$ and $n\ne(kq,kp)$. The closed orbits of $\lambda\pd_{\theta_1}-\mu\pd_{\theta_2}$ have direction $(p,-q)$, and the average of $e^{in\cdot\theta}$ along them vanishes unless $n\cdot(p,-q)=0$, i.e.\ $n\in\Z(q,p)$. Together with $n_1+n_2=k(p+q)$ this forces $n=(kq,kp)$, which is excluded. Hence the orbit average of $f$ is $\Im(\lambda c\,u^k)$. On $T(1)$ this equals $|\lambda c|\,|u|^k\sin(k\vartheta+\arg(\lambda c))$, where $\vartheta=q\theta_1+p\theta_2$ is constant on the orbits; it vanishes on exactly $2k$ orbits.
\end{remark}

\subsection{Holonomy}

\begin{lemma}[straightening the separatrices]\label{lem:separatrices}
Suppose $X=X_0+c\,u^kw\pd_w+R$ with $R=O(|\zeta|^{m+1})$, $m=k(p+q)+1$. The separatrices of $\F$ tangent to the axes are graphs $w=s(z)$ and $z=\sigma(w)$ with $s=O(z^{m+1})$ and $\sigma=O(w^{m+1})$. After the change $(z,w)\mapsto(z-\sigma(w),w-s(z))$, both axes are invariant and $X$ keeps the form $X_0+c\,u^kw\pd_w+O(|\zeta|^{m+1})$. Hence
\begin{equation}\label{eq:straight}
 X=zA\,\pd_z+wB\,\pd_w,\qquad A=\lambda+O(|\zeta|^m),\quad B=-\mu+c\,u^k+O(|\zeta|^m).
\end{equation}
\end{lemma}

\begin{proof}
The separatrices exist and are smooth, by the Briot--Bouquet theorem \cite{Ilyashenko-Yakovenko}. Write $X=(\lambda z+Z_1)\pd_z+(-\mu w+W_1)\pd_w$ with $Z_1=O(|\zeta|^{m+1})$ and $W_1=c\,u^kw+O(|\zeta|^{m+1})$. Invariance of $w=s(z)$ means $\lambda zs'+\mu s=W_1(z,s)-s'Z_1(z,s)$. If $s=s_jz^j+O(z^{j+1})$ with $s_j\ne0$ and $j\le m$, then the left side has the nonzero term $(\lambda j+\mu)s_jz^j$. On the right side, $c\,u^ks=c\,z^{kq}s^{kp+1}=O(z^{j+1})$, the remainder of $W_1$ is $O(z^{m+1})$, and $s'Z_1=O(z^{j+m})$; all are $O(z^{j+1})$. This is a contradiction, so $s=O(z^{m+1})$; similarly $\sigma=O(w^{m+1})$. The change of coordinates differs from the identity by $O(|\zeta|^{m+1})$, so by the last assertion of Lemma~\ref{lem:coordchange} it alters $X$ only by $O(|\zeta|^{m+1})$. The divisibility in~\eqref{eq:straight} expresses the invariance of the axes.
\end{proof}

\begin{proposition}\label{prop:holonomy}
In a chart as in Lemma~\ref{lem:separatrices}, after straightening, let $h_\varepsilon$ be the holonomy of $\{w=0\}$ along the loop $|z|=\varepsilon$ (positively oriented), on the transversal $\{z=\varepsilon\}$ with coordinate $w$. Then $h_\varepsilon'(0)=e^{-2\pi iq/p}$, and
\[
 h_\varepsilon^p(w)=w+\frac{2\pi ipc}{\lambda}\,\varepsilon^{kq}\bigl(1+O(\varepsilon)\bigr)\,w^{kp+1}+O(w^{kp+2}),
\]
with $h_\varepsilon^p-\mathrm{id}=O(w^{kp+1})$, for all small $\varepsilon>0$.
\end{proposition}

\begin{proof}
\emph{Step 1: the truncated field.} For $Y_s=X_0+scu^kw\pd_w$, $u=z^qw^p$ satisfies $d(u^{-k})/d\log z=-kpsc/\lambda$. After $p$ turns around $z=0$ it is translated by $-2\pi ikp^2sc/\lambda$. On $z=1$, where $u=W^p$, the holonomy therefore satisfies
\begin{equation}\label{eq:truncated-holonomy}
 H_s^p(W)=W(1-2\pi ikp^2scW^{kp}/\lambda)^{-1/(kp)}
 =W+\frac{2\pi ipsc}{\lambda}W^{kp+1}+O(W^{2kp+1}).
\end{equation}
The branch tangent to the identity is chosen, and $H_s'(0)=e^{-2\pi iq/p}$.

\emph{Step 2: rescaling.} Put $\zeta=\varepsilon\xi$, $\xi=(Z,W)$, and $X_\varepsilon(\xi)=\varepsilon^{-1}X(\varepsilon\xi)$. By homogeneity, with $\delta=\varepsilon^{k(p+q)}$,
\[
 X_\varepsilon=Y_\delta+\varepsilon^{-1}R(\varepsilon\xi),\qquad \bigl|\varepsilon^{-1}R(\varepsilon\xi)\bigr|\le C\varepsilon^{m}=C\varepsilon\delta\quad\text{for }|\xi|\le2 .
\]
The holonomy $h_\varepsilon$ of the original field is conjugate, by $w=\varepsilon W$, to the holonomy $H^{(\varepsilon)}$ of $X_\varepsilon$ along $|Z|=1$ on $\{Z=1\}$. For the field $Y_\delta$ and $|W_0|\le\rho$ small, the lifted loop stays in $\{|W|\le2\rho\}$, where the $Z$-component of $X_\varepsilon$ is $\lambda Z+O(\varepsilon\delta)$. Gronwall's inequality therefore gives $(H^{(\varepsilon)})^p-H_\delta^p=O(\varepsilon\delta)$ uniformly on $|W|\le\rho$, and the same bound holds for each Taylor coefficient by Cauchy's estimates. Writing $(H^{(\varepsilon)})^p(W)-W=\sum_nE_n(\varepsilon)W^n$, \eqref{eq:truncated-holonomy} gives
\begin{equation}\label{eq:En}
 E_n(\varepsilon)=O(\varepsilon\delta)\ \ (n\le kp),\qquad E_{kp+1}(\varepsilon)=\frac{2\pi ipc}{\lambda}\,\delta+O(\varepsilon\delta).
\end{equation}

\emph{Step 3: transport along the separatrix.} The order $\nu$ of $h_\varepsilon^p-\mathrm{id}$ does not depend on $\varepsilon$, since holonomies on different transversals are conjugate by the transport $\phi$ along $\{w=0\}$. If $\phi(w)=\alpha w+O(w^2)$ maps $\{z=\varepsilon_0\}$ to $\{z=\varepsilon\}$ along the radial path, and $h_{\varepsilon_0}^p=\mathrm{id}+e\,w^\nu+O(w^{\nu+1})$ with $e\ne0$, then $h_\varepsilon^p=\phi\circ h_{\varepsilon_0}^p\circ\phi^{-1}=\mathrm{id}+e\,\alpha^{1-\nu}w^\nu+O(w^{\nu+1})$. By~\eqref{eq:straight}, a solution $v$ of the variational equation along $\{w=0\}$ satisfies $d\log v/d\log z=B(z,0)/A(z,0)=-q/p+O(z^m)$. So $\alpha=(\varepsilon/\varepsilon_0)^{-q/p}\theta(\varepsilon)$ with $\theta$ bounded and bounded away from zero. The same equation integrated around $|z|=\varepsilon$ gives $h_\varepsilon'(0)=e^{-2\pi iq/p}$. Since $w=\varepsilon W$, we get $E_\nu(\varepsilon)=\varepsilon^{\nu-1}e\,\alpha^{1-\nu}$, and therefore
\[
 |E_\nu(\varepsilon)|\asymp\varepsilon^{(\nu-1)(1+q/p)}=\varepsilon^{(\nu-1)(p+q)/p}.
\]

\emph{Step 4: conclusion.} If $\nu\le kp$, then $(\nu-1)(p+q)/p\le k(p+q)-(p+q)/p<k(p+q)+1$. This contradicts $E_\nu(\varepsilon)=O(\varepsilon\delta)=O(\varepsilon^{k(p+q)+1})$ from~\eqref{eq:En}. If $kp+1<\nu<\infty$, then $E_{kp+1}(\varepsilon)=0$, contradicting~\eqref{eq:En} for small $\varepsilon$. If $\nu=\infty$, all the coefficients vanish, giving the same contradiction. Hence $\nu=kp+1$, and the leading coefficient of $h_\varepsilon^p-\mathrm{id}$ is $\varepsilon^{-kp}E_{kp+1}(\varepsilon)=\frac{2\pi ipc}{\lambda}\varepsilon^{kq}(1+O(\varepsilon))$.
\end{proof}

\begin{corollary}\label{cor:holonomy}
Let $h$ be the holonomy of the separatrix tangent to the $z$-eigendirection. Then $h'(0)=e^{-2\pi iq/p}$. If $k=\infty$, then $h^p=\mathrm{id}$; otherwise $\operatorname{ord}_0(h^p-\mathrm{id})=kp+1$, and $h^p$ has $kp$ attracting and $kp$ repelling Leau--Fatou petals. For the holonomy $\tilde h$ of the other separatrix, $\tilde h^q=\mathrm{id}$ if $k=\infty$, and $\operatorname{ord}_0(\tilde h^q-\mathrm{id})=kq+1$ otherwise.
\end{corollary}

\begin{proof}
If $k<\infty$, choose by Proposition~\ref{prop:tau}(3) and Lemma~\ref{lem:separatrices} a chart as in Proposition~\ref{prop:holonomy}. The holonomy of the separatrix on any transversal is analytically conjugate to $h_\varepsilon$, so $\operatorname{ord}(h^p-\mathrm{id})=kp+1$ and $h'(0)=e^{-2\pi iq/p}$. The Leau--Fatou flower theorem \cite{Milnor} gives exactly $kp$ attracting and $kp$ repelling petals. The statement for the other separatrix follows by exchanging the roles of $z$ and $w$. Indeed, $\tau(\F)$ does not distinguish the two separatrices, and in the chart of Proposition~\ref{prop:tau}(3), since $u^kX_0=\lambda u^kz\pd_z-\mu u^kw\pd_w$,
\[
 \Bigl(1+\frac c\mu u^k\Bigr)\bigl(X_0+c\,u^kw\pd_w\bigr)=X_0+\frac{c\lambda}{\mu}\,u^kz\pd_z+O\bigl(|\zeta|^{2k(p+q)+1}\bigr),
\]
and $2k(p+q)+1\ge m+1$. After exchanging $z$ and $w$ and multiplying by $-1$, this is the form of Lemma~\ref{lem:separatrices} with $p$ and $q$ exchanged, so Proposition~\ref{prop:holonomy} gives $\operatorname{ord}(\tilde h^q-\mathrm{id})=kq+1$. If $k=\infty$, then $\F$ is linearizable by Theorem~\ref{thm:H}(a); in linearizing coordinates the holonomy is $w\mapsto e^{-2\pi iq/p}w$, so $h^p=\mathrm{id}$.
\end{proof}

\section{The characteristic flow near the separatrices}\label{sec:characteristic}

Let $X$ generate $\F$ near a sphere $S^3(r)$ and put $g=X(|\zeta|^2)=\ip{X(\zeta)}{\zeta}$. Let $Y$ be the real vector field underlying the complex field $i\,\ol g\,X$, restricted to $S^3(r)$. Since $Y(|\zeta|^2)=2\Re(i\,\ol g\,g)=0$, $Y$ is tangent to the sphere. It vanishes exactly on $M(\F,S^3(r))$ (and at singular points of $X$), and off this set it spans $T\F\cap TS^3(r)$. Replacing $X$ by $vX$ with $v$ a unit replaces $g$ by $vg$ and $Y$ by $|v|^2Y$. So the oriented line field spanned by $Y$, the \emph{characteristic flow}, depends only on $\F$ and the sphere.

\begin{proof}[Proof of Theorem~\ref{thm:char}]
\emph{Transversality.} Let $A\zeta$ be the linear part of $X$ in the chart and let $e$ be a unit eigenvector for the eigenvalue $\lambda$; $S$ is tangent to $\C e$ at $0$. With $X_r(\zeta)=r^{-1}X(r\zeta)$ and $g_r(\zeta)=\ip{X_r(\zeta)}{\zeta}$ on $S^3(1)$, we have $g_r\to g_0=\ip{A\zeta}{\zeta}$ uniformly, and $g_0=\lambda\ne0$ on the Hopf fibre $\pi^{-1}([e])$. Choose a neighbourhood $V$ of this fibre in $S^3(1)$ on which $|g_0|\ge\lambda/2$; then $|g_r|\ge\lambda/4$ on $V$ for small $r$. Since $S$ is a smooth curve tangent to $\C e$, it is transverse to small spheres, and $r^{-1}(S\cap S^3(r))$ is a circle converging in $C^1$ to $\pi^{-1}([e])$. Hence for small $r$ the circle $C_r=S\cap S^3(r)$ lies in $rV$, where $\F$ is transverse to $S^3(r)$. As $C_r$ is a connected component of the intersection of the leaf $S\setminus\{0\}$ with the sphere, it is a closed orbit of the characteristic flow, which is nonsingular on $rV$.

\emph{Conjugacy with the holonomy.} Fix $x_r\in C_r$ and a small real-analytic disc $\Sigma_r\subset S^3(r)$ through $x_r$ transverse to $Y$; then $P_r$ is a real-analytic diffeomorphism germ at $x_r$. At points of $rV$ we have $T\F\cap TS^3(r)=\R Y$, so $T_{x_r}\Sigma_r\cap T_{x_r}\F=0$: $\Sigma_r$ is a real surface transverse to $\F$. Let $\Gamma$ be a holomorphic disc through $x_r$ transverse to $\F$, and take a foliated chart $(a,b)$ near $x_r$ in which the leaves are $\{b=\text{const}\}$ and $\Gamma=\{a=0\}$. The map $\varpi:\Sigma_r\to\Gamma$, $x\mapsto b(x)$, sending a point to the point of $\Gamma$ on the same plaque, is a real-analytic diffeomorphism germ, since $\Sigma_r$ is transverse to the plaques.

Let $\alpha=\Im\sum_j\bar\zeta_j\,d\zeta_j$ be the standard contact form of the sphere, so that $\ker\alpha=T^{\C}S^3(r)$ is its complex tangent line. The field $Y$, with complex velocity $i\bar gX$, satisfies the exact identity
\[
 \alpha(Y)=\Im(i\bar gg)=|g|^2>0
\]
off the tangency set. Let $h_\Gamma$ be the holonomy along $C_r$ oriented by $Y$. This is the positive loop orientation: on the separatrix, parametrized by $\zeta=te+O(t^2)$, the form $\alpha$ is positive on a positively oriented small boundary circle, and $\alpha(Y)>0$. Thus $h_\Gamma$ is analytically conjugate to $h$. Following a nearby $Y$-orbit once around $C_r$, and projecting its endpoints to $\Gamma$ along plaques, computes exactly this holonomy. Equivalently, a chain of foliated charts along $C_r$ gives
\[
 \varpi\circ P_r=h_\Gamma\circ\varpi
\]
by the definition of holonomy \cite{Camacho-LN}. This proves the real-analytic conjugacy.

For orientation, recall that the boundary orientation of the sphere is given by $\alpha\wedge d\alpha$, and that $d\alpha$ induces the complex orientation on the complex tangent line $\ker\alpha$. Off the tangency set, $\ker\alpha$ is a complex line different from $T\F$, so the projection $\C^2\to\C^2/T\F$ restricts to a complex-linear isomorphism on it. Since $\alpha(Y)>0$, projection along $Y$ therefore identifies the sphere-and-flow orientation of $\Sigma_r$ with the transverse complex orientation of $\F$. Thus $\varpi$ preserves orientation, proving (b). The same projection identifies the transverse complex structure induced by $\ker\alpha$ with that of $\F$, so $P_r$ is holomorphic for it.

\emph{Parts (a) and (b).} Real-analytic conjugacy preserves the properties in question. By Corollary~\ref{cor:holonomy}, $h'(0)=e^{-2\pi iq/p}$ is a primitive $p$-th root of unity, so no iterate $h^j$ with $0<j<p$ is the identity. If $k=\infty$, then $h^p=\mathrm{id}$, hence $P_r^p=\mathrm{id}$ and $P_r^j\ne\mathrm{id}$ for $0<j<p$; every orbit of $P_r$ other than $x_r$ has exact period $p$, so the corresponding characteristic orbit closes up after winding $p$ times along $C_r$.

If $k<\infty$, then $h^p-\mathrm{id}$ has order exactly $kp+1$ by Corollary~\ref{cor:holonomy}, so $0$ is an isolated fixed point of $h^p$. Its Lefschetz index equals the local degree of $w\mapsto h^p(w)-w$, which is $kp+1$ because this map is holomorphic with a zero of order $kp+1$. The Lefschetz index of an isolated fixed point is invariant under topological conjugacy, so $x_r$ is an isolated fixed point of $P_r^p$ of index $kp+1$. By the Leau--Fatou flower theorem \cite{Milnor}, $h^p$ has exactly $kp$ attracting and $kp$ repelling petals, alternating around $0$. Their images under $\varpi^{-1}$ are petals of $P_r^p$ with the same properties. If $x$ lies in an attracting petal of $P_r^p$, then $P_r^{pn}(x)\to x_r$, so the characteristic orbit of $x$ accumulates on $C_r$ in forward time; for repelling petals the same holds in backward time.

The statement for the other separatrix follows in the same way from the order of $\tilde h^q-\mathrm{id}$ in Corollary~\ref{cor:holonomy}. Finally, the index $kp+1$ of $P_r^p$ at $x_r$ determines $k$, and $k=\infty$ if and only if $P_r^p=\mathrm{id}$, which by Theorem~\ref{thm:H}(a) holds if and only if $\F$ is analytically linearizable.
\end{proof}

\begin{lemma}\label{lem:ra-conj}
Let $f,g:(\C,0)\to(\C,0)$ be holomorphic germs fixing $0$, and let $H$ be a real-analytic diffeomorphism germ of $(\C,0)$ with $H\circ f=g\circ H$. Put $f^*(\zeta)=\ol{f(\ol\zeta)}$. If $H$ preserves orientation, then $f$ and $g$ are holomorphically conjugate. If $H$ reverses orientation, then $f^*$ and $g$ are holomorphically conjugate.
\end{lemma}

\begin{proof}
Write the convergent complexification of $H$ as
\[
 \mathcal H(z,\zeta)=\sum_{m,n\ge0}a_{mn}z^m\zeta^n,
 \qquad H(z)=\mathcal H(z,\ol z),
\]
with $a_{00}=0$. The two sides of
\[
 \mathcal H\bigl(f(z),f^*(\zeta)\bigr)=g\bigl(\mathcal H(z,\zeta)\bigr)
\]
are holomorphic near $(0,0)\in\C^2$ and agree on the totally real plane $\{\zeta=\ol z\}$, hence agree identically. The real Jacobian determinant of $H$ at $0$ is $|a_{10}|^2-|a_{01}|^2\ne0$. If $H$ preserves orientation, then $a_{10}\ne0$. Since $f^*(0)=0$, setting $\zeta=0$ gives the holomorphic diffeomorphism germ $\phi(z)=\mathcal H(z,0)$, with $\phi\circ f=g\circ\phi$. If $H$ reverses orientation, then $a_{01}\ne0$, and setting $z=0$ gives the holomorphic diffeomorphism germ $\psi(\zeta)=\mathcal H(0,\zeta)$, with $\psi\circ f^*=g\circ\psi$.
\end{proof}

For non-periodic resonant germs, i.e.\ resonant germs that are not of finite order, Martinet and Ramis prove much more: an orientation-preserving $C^1$ conjugacy is itself biholomorphic \cite[Ch.~III, Th\'eor\`eme~5.2]{MartinetRamis}. The restriction is necessary: for $f(z)=e^{-2\pi iq/p}z$, the real-analytic diffeomorphism $H(z)=z(1+|z|^2)$ satisfies $H\circ f=f\circ H$ but is not holomorphic. Their proof uses the sectorial normalizations of the germs, so it does not extend to non-linearizable germs with irrational multiplier, for which Lemma~\ref{lem:ra-conj} is all we use.

\begin{proof}[Proof of Corollary~\ref{cor:rigid}]
The transversality and conjugacy steps of the proof of Theorem~\ref{thm:char} hold for either rational or irrational Siegel ratio. The orbital equivalence $\Phi$ maps a transversal $\Sigma_r$ to a transversal of the same regularity and conjugates their return maps. Projecting $\Phi(\Sigma_r)$ to a chosen target transversal along the target flow gives a conjugacy between $P_r$ and $P'_{r'}$. It preserves transverse orientation because $\Phi$ preserves the sphere orientation and time orientation. Theorem~\ref{thm:char} therefore gives an orientation-preserving conjugacy of the same regularity between $h$ and $h'$. In case (b), the multiplier $e^{-2\pi iq/p}$ of $h$ is a root of unity. If $h$ has finite order, then so has $h'$, since conjugacy preserves the order. Two holomorphic germs of finite order are holomorphically linearizable, and $h,h'$ have the same multiplier because the ratios agree, so they are holomorphically conjugate. If $h$ has infinite order, so has $h'$, and Martinet--Ramis \cite[Ch.~III, Th\'eor\`emes~5.1--5.2]{MartinetRamis} gives holomorphic conjugacy. In case (a), Lemma~\ref{lem:ra-conj} gives it. Since the eigenvalue ratios agree, \cite[Ch.~II, Proposition~2.1]{MartinetRamis}, which applies to every positive real saddle ratio in the form $x\,dy+\xi y(1+\cdots)\,dx$, gives analytic equivalence of the saddles.
\end{proof}

\begin{remark}[what the characteristic flow determines]\label{rem:CR}
Together with the CR structure of the sphere, the characteristic flow determines the foliation \cite{Alexander}; Corollary~\ref{cor:rigid} uses no CR structure. Its regularity at the separatrix circle is essential: for fixed ratio, Martinet--Ramis identify the $C^\infty$ leaf-space type away from the singularity by $k$ \cite[p.~577]{MartinetRamis}, and topologically conjugate parabolic germs admit a conjugacy that is $C^\infty$ off the fixed point \cite[proof of Theorem~1.3]{Bracci}. This differs from real one-dimensional parabolic dynamics, whose smooth classification is formal \cite{Takens}.
\end{remark}

\begin{remark}[the irrational case]\label{rem:irrational}
Let $\F$ be a Siegel saddle with $\lambda/\mu\notin\mathbb Q$. The transversality and conjugacy steps of the proof of Theorem~\ref{thm:char} use only that $\lambda\neq0$, so for small $r$ the separatrix circle $C_r$ is again a closed orbit of the characteristic flow whose first-return map $P_r$ is real-analytically conjugate to $h$, where $h(w)=e^{-2\pi i\mu/\lambda}w+\cdots$. A germ of this form is analytically linearizable if and only if $0$ is a Lyapunov stable fixed point \cite[Proposition~4.1]{Abate}, and stability is invariant under topological conjugacy. By Mattei--Moussu \cite[Ch.~II, Th\'eor\`eme~2]{MatteiMoussu}, $\F$ is linearizable if and only if $h$ is. Hence $\F$ is analytically linearizable if and only if $C_r$ is Lyapunov stable for the characteristic flow, for some (equivalently, every) analytic chart and small $r$. When $\F$ is not linearizable, the hedgehogs of P\'erez-Marco \cite{PerezMarco} give compact connected invariant sets of $P_r$ accumulating on $x_r$. Moreover, Corollary~\ref{cor:rigid}(a) applies: the characteristic flow near $C_r$, up to real-analytic orbital equivalence, determines the analytic class of $\F$ among saddles with the same irrational ratio. Thus the characteristic flow detects linearizability also in the irrational case.
\end{remark}

\section{The characteristic flow on the whole sphere}\label{sec:whole}

For a holomorphic vector field $X=P\pd_z+Q\pd_w$ we write
\[
 g=\ip{X(\zeta)}{\zeta},\qquad \mathsf S=\ol X(g)=|P|^2+|Q|^2,\qquad \mathsf A=X(g)=(XP)\ol z+(XQ)\ol w,\qquad \mathscr D=|\mathsf A|^2-\mathsf S^2,
\]
and $g_W(\zeta)=\ip{W(\zeta)}{\zeta}$ for a vector field $W$. As in Section~\ref{sec:characteristic}, $Y$ is the real vector field underlying $i\,\ol g\,X$ on a sphere. If $X$ is replaced by $vX$, $v$ a unit, then at tangency points $\mathsf S$ and $\mathsf A$ are multiplied by $|v|^2$ and $v^2$, so $\mathscr D$ is multiplied by $|v|^4$.

Let $\F$ be a resonant saddle with invariant $k<\infty$ and put $m=k(p+q)+1$. Any chart as in Theorem~\ref{thm:H}(b) has a generator of the form~\eqref{eq:normalized-generator}, after multiplication by a unit; we call it the normalized generator, and we let $\Gamma\subset T(1)$ be the union of the $2k$ circles $\{\Im(cu^k)=0\}$. By Theorem~\ref{thm:C}(b), $\Gamma$ is the limit of the rescaled tangency loci $r^{-1}M(\F,S^3(r))$. We prove:

\begin{theorem}[transverse type of the contacts]\label{thm:transverse}
Let $\F$ be a resonant saddle with invariant $k<\infty$, in a chart as in Theorem~\ref{thm:H}(b), with normalized generator~\eqref{eq:normalized-generator}.
\begin{enumerate}[label=\textup{(\alph*)}]
\item At a tangency point $\zeta\in S^3(r)$ with $\mathscr D(\zeta)>0$ (resp.\ $<0$), the function $|\cdot|^2$ restricted to the leaf through $\zeta$ has a nondegenerate saddle (resp.\ a strict local minimum) at $\zeta$. Near $\zeta$, the tangency set is a curve of zeros of the characteristic field, transverse to the leaf, whose transverse linearization has eigenvalues $\pm\sqrt{\mathscr D}$ (saddle type, resp.\ centre type). In the centre case, closed characteristic orbits accumulate on $\zeta$.
\item Uniformly at the tangency points $\zeta=r\xi$ of $S^3(r)$, as $r\to0$,
\[
 \mathscr D(\zeta)=2\lambda\mu\,r^4\Bigl(\sum_{d=m+1}^{2m-1}r^{d-1}\,\Re\ip{[X_0,X_d](\xi)}{\xi}\;+\;kp\,t_0\,|c|^2(s_0^qt_0^p)^k\,r^{2m-2}\Bigr)+O(r^{2m+3}),
\]
where $s_0=\mu/(\lambda+\mu)$, $t_0=\lambda/(\lambda+\mu)$ and $[X_0,X_d]=DX_d\cdot X_0-DX_0\cdot X_d$. Resonant terms do not contribute, since $[X_0,X_d]=0$ for them.
\item In a Poincar\'e--Dulac chart of order $2m-1$, for all small $r$,
\[
 \mathscr D\ \ge\ \lambda\mu\,kp\,t_0\,|c|^2(s_0^qt_0^p)^k\,r^{2m+2}>0
\]
on the tangency set. Hence each of the $2k$ tangency circles is a normally hyperbolic circle of zeros of the characteristic field, of saddle type, and no contact is elliptic. If instead the chart is Poincar\'e--Dulac of order $d-1$ for some $m+1\le d\le 2m-2$, and $\tau_d=\Re\ip{[X_0,X_d](\xi)}{\xi}$ does not vanish identically on $\Gamma$, then along the tangency circles
\[
 \mathscr D(\zeta)=2\lambda\mu\,r^{d+3}\,\tau_d(\xi_0)+O(r^{d+4}),
\]
where $\xi_0\in\Gamma$ lies at distance $O(r)$ from $\zeta/r$. In particular, on every compact subarc of $\Gamma$ on which $\tau_d<0$, the corresponding contacts are elliptic for all sufficiently small $r$, and they carry families of closed characteristic orbits. If the chart is Poincar\'e--Dulac of order exactly $2m-2$, a non-resonant term $X_{2m-1}$ contributes to~(b) at the same order $r^{2m+2}$ as the resonant leading term, and the sign of $\mathscr D$ depends on both.
\end{enumerate}
\end{theorem}

Part~(c) contains Theorem~\ref{thm:H}(c). The following example shows that the second case of~(c) occurs.

\begin{exe}[elliptic arcs]\label{ex:elliptic}
Let $X_a=z\pd_z+w(-1+zw+az^3)\pd_w$ with $a>0$. Here $p=q=k=1$ and $m=3$; the chart is Poincar\'e--Dulac of order $3$ but not of order $4$, since $X_4=az^3w\pd_w$ is nonresonant, with $[X_0,X_4]=3X_4$. The second case of~(c) gives
\[
 \mathscr D=\tfrac{3a}{2\sqrt2}\,r^7\cos3\theta+O(r^8),\qquad \theta=\arg z,
\]
along each of the two tangency circles, so each circle consists of three arcs of saddle type and three of centre type. On the other hand, with $J(z)=\sum_{n\ge1}\frac{(a/3)^nz^{3n}}{3n\,n!}$, so that $zJ'(z)=e^{az^3/3}-1$, the coordinates $Z=ze^{J(z)}$, $W=we^{-az^3/3-J(z)}$ satisfy $ZW=zwe^{-az^3/3}$, $X_a(Z)=Ze^{az^3/3}$ and $X_a(W)=We^{az^3/3}(-1+ZW)$. Hence they transform $X_a$ into $e^{az^3/3}\bigl(Z\pd_Z+W(-1+ZW)\pd_W\bigr)$, which is Poincar\'e--Dulac of every order, and there all contacts are of saddle type. The transverse type of the contacts thus depends on the chart, even among charts with the same leading resonant term.
\end{exe}

\subsection{The meaning of \texorpdfstring{$\mathscr D$}{D}}

\begin{lemma}\label{lem:meaning}
Let $x\in S^3(r)$ be a tangency point, $L$ the leaf through $x$, and $t$ the complex time of $X$ on $L$, so that $X=\pd_t$ on $L$.
\begin{enumerate}[label=\textup{(\arabic*)}]
\item The function $\rho=|\zeta|^2$ restricted to $L$ has a critical point at $x$, and $\rho|_L=r^2+\mathsf S|t|^2+\Re(\mathsf At^2)+O(|t|^3)$. It is a nondegenerate saddle if $\mathscr D>0$ and a strict minimum if $\mathscr D<0$.
\item If $\mathscr D(x)\ne0$, then $g|_{S^3(r)}$ is a submersion at $x$, and near $x$ the tangency set is a smooth curve transverse to $T_xL$. On $L$, $Y$ is the Hamiltonian vector field of $\frac12\rho|_L$ with respect to the area form of the coordinate $t$. The linearization of $Y$ at $x$ vanishes along the tangency curve and has eigenvalues $\pm\sqrt{\mathscr D}$ on $T_xL$.
\item If $\mathscr D(x)<0$, then closed characteristic orbits on $S^3(r)$ accumulate on $x$.
\end{enumerate}
\end{lemma}

\begin{proof}
(1) Along $L$ we have $\pd_t\rho=g$, $\pd_{\bar t}\rho=\ol g$, $\pd_t\pd_{\bar t}\rho=\ol X(g)=\mathsf S$ and $\pd_t^2\rho=X(g)=\mathsf A$, and $g(x)=0$. The real Hessian of $\mathsf S|t|^2+\Re(\mathsf At^2)$ has eigenvalues $2(\mathsf S\pm|\mathsf A|)$.

(2) The real tangent vector $v$ of $L$ corresponding to $\tau X(x)$, $\tau\in\C$, satisfies $dg(v)=\tau\mathsf A+\bar\tau\mathsf S$. This vanishes for some $\tau\neq0$ only if $|\mathsf A|=\mathsf S$. Hence, if $\mathscr D\ne0$, $dg$ maps $T_xL\subset T_xS^3(r)$ isomorphically onto $\C$; so $g|_{S^3(r)}$ is a submersion at $x$ and its zero set is a curve transverse to $T_xL$. On $L$, with $t=a+ib$, the velocity $i\,\ol g=i\,\pd_{\bar t}\rho$ is $\frac12(-\rho_b+i\rho_a)$, which is the Hamiltonian field of $\frac12\rho$. Its linearization at the critical point $x$ has eigenvalues $\nu$ with $\nu^2=-\det\operatorname{Hess}(\tfrac12\rho|_L)=|\mathsf A|^2-\mathsf S^2$. Since $Y$ vanishes along the tangency curve and is tangent to the leaves, its linearization at $x$ vanishes on the tangent line of the curve and preserves $T_xL$.

(3) By (1), $\rho|_L$ has a nondegenerate local minimum at $x$, with value $r^2$. By the Morse lemma with parameters \cite{AGV}, the leaves through points $y$ near $x$ have nearby nondegenerate local minima of $\rho$, depending continuously on $y$. If $\rho(y)<r^2$, the minimum is $<r^2$, and the leaf meets $S^3(r)$ near $x$ in a small circle, a regular level set of a Morse function near its minimum. These circles are closed orbits of $Y$, and they accumulate on $x$.
\end{proof}

\begin{remark}\label{rem:morse-type}
By Lemma~\ref{lem:meaning}(1), $\mathscr D(x)\ne0$ means that $x$ is a nondegenerate critical point of the distance function on its leaf, i.e.\ that $\F$ is of Morse type at $x$ in the sense of Ito and Sc\'ardua \cite{ItoScardua2010}, with Morse index $1$ if $\mathscr D>0$ and $0$ if $\mathscr D<0$. Lemma~\ref{lem:meaning} translates the Morse index into the dynamics of the characteristic flow: index $1$ gives saddle circles and index $0$ gives centres surrounded by closed characteristic orbits. In this language, Theorem~\ref{thm:transverse}(c) says that in Poincar\'e--Dulac charts of order $2m-1$ the foliation is of Morse type at every contact with small spheres, all contacts having index $1$. In Example~\ref{ex:elliptic} the contacts are of Morse type except at the points where $\mathscr D$ changes sign, close to the six points of each tangency circle where $\cos3\theta=0$; there elliptic and hyperbolic arcs meet. These belong to the degeneracy locus of the contact variety studied in \cite{ItoScarduaYamagishi2014}, where contacts are related to fixed points of an anti-holomorphic dynamics and degenerate contacts to bifurcations of neutral fixed points.
\end{remark}

\subsection{An exact identity}

\begin{lemma}\label{lem:identity}
Let $X=X_0+V$ with $V=(P_V,Q_V)$ holomorphic, and put $S_0(\zeta)=\lambda^2|z|^2+\mu^2|w|^2$, $\delta\mathsf A=\mathsf A-S_0$, $\delta\mathsf S=\mathsf S-S_0$. Then, at every point,
\[
 \delta\mathsf A=g_{[X_0,V]}+2(\lambda P_V\ol z-\mu Q_V\ol w)+g_{DV\cdot V},\qquad \delta\mathsf S=2\Re(\lambda P_V\ol z-\mu Q_V\ol w)+|V|^2,
\]
\[
 \mathscr D=2S_0\Bigl(\Re g_{[X_0,V]}+\Re g_{DV\cdot V}-|V|^2\Bigr)+(\Re\delta\mathsf A)^2-(\delta\mathsf S)^2+(\Im\delta\mathsf A)^2 .
\]
In particular, $\mathscr D=0$ identically for $X_0$, and $\mathscr D=2S_0\Re g_{[X_0,V]}+O(\|V\|_{C^1}^2)$.
\end{lemma}

\begin{proof}
For $X_0$ itself, $\mathsf A=\mathsf S=S_0$. Expanding $(X_0+V)(\lambda z+P_V)=\lambda^2z+\lambda P_V+X_0P_V+VP_V$ and $(X_0+V)(-\mu w+Q_V)=\mu^2w-\mu Q_V+X_0Q_V+VQ_V$, and using $[X_0,V]=(X_0P_V-\lambda P_V,\ X_0Q_V+\mu Q_V)$, gives $\delta\mathsf A$. The formula for $\delta\mathsf S$ is the expansion of $|\lambda z+P_V|^2+|-\mu w+Q_V|^2$. The identity for $\mathscr D$ is $(S_0+\Re\delta\mathsf A)^2+(\Im\delta\mathsf A)^2-(S_0+\delta\mathsf S)^2$, expanded; the terms $2\Re(\lambda P_V\ol z-\mu Q_V\ol w)$ cancel.
\end{proof}

By Lemma~\ref{lem:homological}(1), $[X_0,\cdot\,]$ acts diagonally on monomial fields and vanishes exactly on the resonant fields. Hence, to first order, resonant terms do not change $\mathscr D$.

\subsection{The model is exactly hyperbolic}

\begin{proposition}\label{prop:model-exact}
Let $X=X_0+c\,u^kw\pd_w$, put $e=c\,u^k$ and $f=-\mu+e$. At every point with $zw\ne0$,
\[
 \mathsf A=\lambda^2|z|^2+|w|^2(f^2+kp\,e^2),\qquad \mathsf S=\lambda^2|z|^2+|w|^2|f|^2 .
\]
At every tangency point, $e$ is real and nonzero, and $\mathscr D=kp\,|w|^2e^2(\mathsf A+\mathsf S)>0$.
\end{proposition}

\begin{proof}
We have $XP=\lambda^2z$. Since $\pd_zu=qu/z$ and $\pd_wu=pu/w$, we get $\pd_zQ=kq\,ew/z$, $\pd_wQ=f+kp\,e$, and so
\[
 XQ=\lambda z\cdot kq\,\frac{ew}{z}+wf(f+kp\,e)=w\bigl(f^2+ke(\lambda q+pf)\bigr)=w\bigl(f^2+kp\,e^2\bigr),
\]
because $\lambda q+pf=\lambda q-p\mu+pe=pe$. This gives $\mathsf A$, and $\mathsf S$ is immediate. At a tangency point, $g=\lambda|z|^2+|w|^2f=0$ forces $f$, hence $e$, to be real. Moreover $e\ne0$, since otherwise $\lambda|z|^2=\mu|w|^2$, so $zw\ne0$ and $u\ne0$. Then $\mathsf A$ and $\mathsf S$ are real and positive, $\mathsf A-\mathsf S=kp|w|^2e^2$, and $\mathscr D=(\mathsf A-\mathsf S)(\mathsf A+\mathsf S)$.
\end{proof}

\subsection{Proof of Theorem~\ref{thm:transverse}}

Part (a) is Lemma~\ref{lem:meaning}.

(b) Write the normalized generator as $X=X_0+V_m+W$ with $V_m=c\,u^kw\pd_w$ and $W=\sum_{d\ge m+1}X_d$. For $\xi$ near $S^3(1)$ put
\[
 X^{(r)}(\xi)=r^{-1}X(r\xi)=M_r+W_r,\qquad M_r=X_0+r^{m-1}V_m,\qquad W_r=\sum_{d\ge m+1}r^{d-1}X_d .
\]
For $\xi\in S^3(1)$, the quantities $g$, $\mathsf S$ and $\mathsf A$ of $X$ at $r\xi$ are $r^2$ times those of $X^{(r)}$ at $\xi$. Hence $\mathscr D_X(r\xi)=r^4\mathscr D_{X^{(r)}}(\xi)$, and the tangency points of $X$ on $S^3(r)$ are $r$ times those of $X^{(r)}$ on $S^3(1)$. On a fixed neighbourhood $B$ of $S^3(1)$,
\[
 \|r^{m-1}V_m\|_{C^2(B)}=O(r^{m-1}),\quad \|W_r\|_{C^2(B)}=O(r^m),\quad \Bigl\|W_r-\sum_{d=m+1}^{2m-1}r^{d-1}X_d\Bigr\|_{C^2(B)}=O(r^{2m-1}).
\]
For a holomorphic field $Z$, $\mathscr D_Z(\xi)$ is a fixed real polynomial $\Phi(\xi,Z(\xi),DZ(\xi))$. Taylor's formula at $Z=M_r$, together with $d\Phi_{M_r}=d\Phi_{X_0}+O(r^{m-1})$ and Lemma~\ref{lem:identity}, which gives $d\Phi_{X_0}[W]=2S_0\Re g_{[X_0,W]}$, yields uniformly on $B$
\begin{equation}\label{eq:D-split}
 \mathscr D_{X^{(r)}}=\mathscr D_{M_r}+2S_0\sum_{d=m+1}^{2m-1}r^{d-1}\Re g_{[X_0,X_d]}+O(r^{2m-1}).
\end{equation}
Now let $\xi$ be a tangency point of $X^{(r)}$. By Lemma~\ref{lem:reduction}, $|\xi_1|^2=s_0+O(r^{m-1})$; in particular $|\xi_2|$ and $|u(\xi)|$ are bounded below. Put $e=r^{m-1}c\,u^k(\xi)$ and $f=-\mu+e$. Since $g_{X^{(r)}}(\xi)=0$, we have $g_{M_r}(\xi)=-g_{W_r}(\xi)=O(r^m)$. As $g_{M_r}=\lambda|z|^2+|w|^2f$, this gives $\Im e=O(r^m)$, while $|e|\asymp r^{m-1}$. The formulas of Proposition~\ref{prop:model-exact}, applied to $M_r$ (with $r^{m-1}c$ in place of $c$), give at $\xi$
\[
 \mathsf A-\mathsf S=|w|^2\bigl(f^2-|f|^2+kp\,e^2\bigr)=|w|^2\bigl(kp\,e^2+2if\,\Im e\bigr).
\]
Hence $\Re(\mathsf A-\mathsf S)=kp|w|^2(\Re e)^2+O(r^{2m})$ and $|\mathsf A-\mathsf S|^2=O(r^{2m})$. Since $\mathscr D_{M_r}=2\mathsf S\Re(\mathsf A-\mathsf S)+|\mathsf A-\mathsf S|^2$, $\mathsf S=S_0+O(r^{m-1})$ and $(\Re e)^2=r^{2m-2}|c|^2|u|^{2k}+O(r^{2m})$, we get
\[
 \mathscr D_{M_r}(\xi)=2S_0(\xi)\,|\xi_2|^2\,kp\,|c|^2|u(\xi)|^{2k}r^{2m-2}+O(r^{2m})+O(r^{3m-3}).
\]
Finally $S_0$, $|\xi_2|^2$ and $|u|$ depend only on $|\xi_1|^2=s_0+O(r^{m-1})$, and at $|\xi_1|^2=s_0$ they equal $\lambda\mu$, $t_0$ and $(s_0^qt_0^p)^{1/2}$. Replacing them by these values, and $S_0(\xi)$ by $\lambda\mu$ in~\eqref{eq:D-split}, costs $O(r^{2m-1})$, since $m\ge3$. Multiplying by $r^4$ gives (b).

(c) In a Poincar\'e--Dulac chart of order $2m-1$, $[X_0,X_d]=0$ for $m+1\le d\le2m-1$, and (b) gives $\mathscr D=2\lambda\mu kp\,t_0|c|^2(s_0^qt_0^p)^kr^{2m+2}+O(r^{2m+3})$. By (a), the tangency circles are curves of zeros of $Y$ whose linearization has eigenvalues $0$ along the curve and $\pm\sqrt{\mathscr D}\ne0$ on a complementary plane; that is, they are normally hyperbolic of saddle type. In the second case, $[X_0,X_{d'}]=0$ for $m+1\le d'<d$, so the leading term in~(b) is $r^{d-1}\Re g_{[X_0,X_d]}(\xi)$. By Theorem~\ref{thm:C}(b) and Lemma~\ref{lem:reduction}, $\xi$ lies at distance $O(r)$ from a point $\xi_0\in\Gamma$, so $\Re g_{[X_0,X_d]}(\xi)=\tau_d(\xi_0)+O(r)$. The other terms are $O(r^d)$, since $d\le2m-2$. On a compact subarc where $\tau_d\le-\delta<0$, the error is smaller than $\lambda\mu\delta r^{d+3}$ for small $r$, so $\mathscr D<0$ there, and Lemma~\ref{lem:meaning}(3) gives the closed orbits.\qed

\subsection{Proof of Theorem~\ref{thm:model}}

\emph{Step 0: normalization.} Up to a positive constant change of time, the homothety $\xi\mapsto r\xi$ identifies the characteristic flow of $X$ on $S^3(r)$ with that of $r^{-1}X(r\xi)=X_0+r^{m-1}c\,u^kw\pd_w$ on $S^3(1)$. A rotation $z\mapsto e^{i\alpha}z$ is unitary, commutes with $X_0$, and multiplies $c$ by $e^{-ikq\alpha}$. We may therefore assume $r=1$ and $c>0$, the hypothesis becoming $c<e\mu$. We write $s=|z|^2$ on $S^3(1)$, $C=C_1$ and $C'=C'_1$.

\emph{Step 1: proof of (a).} Each $R_\theta$ is unitary, preserves $u$, and commutes with $X_0$ and with $u^kw\pd_w$, so it preserves $X$ and $S^3(1)$. Hence $g\circ R_\theta=g$ and $R_\theta$ maps $Y$ to $Y$. The multivalued function $H=\frac1{c\,u^k}+\frac{kp}{\lambda}\log z$ satisfies $X(H)=0$, since $X(u)=pc\,u^{k+1}$ and $X(\log z)=\lambda$. Its branches differ by imaginary constants, so $K=\Re H$ is a single-valued first integral of $\F$ off the axes, hence of the characteristic flow.

\emph{Step 2: the quotient.} The action $R$ is free off the axes, since $\gcd(p,q)=1$. Its isotropy groups on $C$ and $C'$ are $\Z/p$ and $\Z/q$. The functions $s$ and $\varphi=q\arg z+p\arg w$ are invariant and identify $(S^3(1)\setminus(C\cup C'))/R$ with the annulus $\mathcal A=(0,1)\times\R/2\pi\Z$; the ends $s\to1$ and $s\to0$ correspond to $C$ and $C'$. Since $|u|=s^{q/2}(1-s)^{p/2}$ and $\arg u=\varphi$, the function $K$ descends to
\[
 \ol K(s,\varphi)=\frac{\cos(k\varphi)}{c}\,F(s)+\frac{kp}{2\lambda}\log s,\qquad F(s)=s^{-kq/2}(1-s)^{-kp/2}.
\]
The field $Y$ descends to a complete vector field $\ol Y$ on $\mathcal A$, tangent to the level sets of $\ol K$. Moreover $\ol Y$ vanishes exactly at the images of the tangency points. Indeed, if $Y(x)\ne0$ is tangent to the $R$-orbit, then $T_x\F$ contains the generator $i(pz,-qw)=\frac{iq}{\mu}X_0(x)$ of the action; hence $X_0\wedge X=\lambda c\,zwu^k$ vanishes at $x$, so $x$ lies on an axis.

\emph{Step 3: $\ol K$ along the rays $\varphi=\mathrm{const}$.} Put $s_0=q/(p+q)=\mu/(\lambda+\mu)$ and $\beta(s)=(p+q)F(s)(s-s_0)/(1-s)$. A direct computation gives
\[
 \pd_s\ol K=\frac{k}{2s}\Bigl(\frac{\cos(k\varphi)}{c}\,\beta(s)+\frac p\lambda\Bigr),\qquad \pd_\varphi\ol K=-\frac{k\sin(k\varphi)}{c}F(s).
\]
The function $\beta$ is a strictly increasing bijection of $(0,1)$ onto $\R$. For $s>s_0$ every factor of $\beta$ is positive and $(\log\beta)'=\frac{k(p+q)(s-s_0)}{2s(1-s)}+\frac1{s-s_0}+\frac1{1-s}>0$. For $s<s_0$, $(\log|\beta|)'=\frac{k(p+q)(s-s_0)}{2s(1-s)}+\frac{1-s_0}{(s-s_0)(1-s)}<0$. Moreover $\beta'(s_0)>0$, and $\beta\to\mp\infty$ as $s\to0,1$. Consequently, on each ray with $\cos(k\varphi)\ne0$ there is a unique critical point $s^*(\varphi)$ of $\ol K(\cdot,\varphi)$:
\begin{itemize}
\item if $\cos(k\varphi)>0$ (a \emph{valley} ray), $\ol K$ decreases on $(0,s^*)$, increases on $(s^*,1)$, and tends to $+\infty$ at both ends; let $\mathsf m(\varphi)$ be its minimum;
\item if $\cos(k\varphi)<0$ (a \emph{ridge} ray), $\ol K$ increases, then decreases, and tends to $-\infty$ at both ends; let $\mathsf M(\varphi)$ be its maximum;
\item if $\cos(k\varphi)=0$ (an \emph{edge}), $\ol K=\frac{kp}{2\lambda}\log s$ increases from $-\infty$ to $0$.
\end{itemize}
As $\cos(k\varphi)\to0^+$, $\beta(s^*)=-pc/(\lambda\cos(k\varphi))\to-\infty$, so $s^*\to0$, $\cos(k\varphi)F(s^*)/c$ stays bounded and $\mathsf m\to-\infty$. As $\cos(k\varphi)\to0^-$, $s^*\to1$, $\cos(k\varphi)F(s^*)\to0$ and $\mathsf M\to0^-$. By the envelope formula, $\mathsf m'(\varphi)=\mathsf M'(\varphi)=\pd_\varphi\ol K(s^*(\varphi),\varphi)$. Hence on each valley sector $\mathcal V$ (a component of $\{\cos k\varphi>0\}$) $\mathsf m$ increases to its value $K_V$ at the centre and then decreases. On each ridge sector $\mathcal R$, $\mathsf M$ decreases to its value $K_R$ at the centre and then increases.

\emph{Step 4: critical points.} By Step~3, $d\ol K=0$ exactly at the $2k$ points $P_j=(s^*(j\pi/k),j\pi/k)$, $j=0,\dots,2k-1$. There $\pd_s\pd_\varphi\ol K=0$, while $\pd_\varphi^2\ol K=-k^2\cos(k\varphi)F/c$ and $\pd_s^2\ol K=\frac{k}{2s}\cos(k\varphi)\beta'(s)/c$ have opposite signs; so each $P_j$ is a nondegenerate saddle. The equation $\pd_s\ol K=0$ at $\varphi=j\pi/k$ reads $(\lambda+\mu)(s-s_0)=-c\cos(j\pi)(1-s)|u|^k$, because $\lambda(p+q)/p=\lambda+\mu$. This is exactly the tangency equation $\lambda s+(1-s)(-\mu+cu^k)=0$ with $cu^k$ real. Hence the tangency set is $\bigcup_jM_j$ with $M_j$ the $R$-orbit over $P_j$, in accordance with Step~2. By symmetry under $\varphi\mapsto\varphi+2\pi/k$ and $\varphi\mapsto-\varphi$, all valley saddles have the value $K_V$ and all ridge saddles the value $K_R$. Clearly $K_R<0$, since $\ol K<0$ on ridge rays. Further, $K_V>0$: on the valley centre ray $\ol K\ge F/c-\frac{kp}{2\lambda}|\log s|$, and $|\log s|\le\frac{2}{ekq}s^{-kq/2}\le\frac{2}{ekq}F$, so $\ol K\ge F\bigl(\frac1c-\frac{p}{e\lambda q}\bigr)=F\bigl(\frac1c-\frac1{e\mu}\bigr)>0$.

\emph{Step 5: level sets.} Fix a value $a\ne K_V,K_R$. By Step~3, on a valley ray $\{\ol K=a\}$ has two points $s_-(\varphi)<s^*(\varphi)<s_+(\varphi)$ if $a>\mathsf m(\varphi)$, and none if $a<\mathsf m(\varphi)$. On a ridge ray it has two points if $a<\mathsf M(\varphi)$, and none if $a>\mathsf M(\varphi)$. On an edge it has one point $s_a$, with $\frac{kp}{2\lambda}\log s_a=a$, if $a<0$, and none if $a\ge0$. The branches $s_\pm$ are real-analytic, since $\pd_s\ol K\ne0$ there, and they join at $s^*$ where $a=\mathsf m$ or $a=\mathsf M$. At the edges of a valley sector, $s_-\to0$, while $s_+\to s_a$ if $a<0$ and $s_+\to1$ if $a\ge0$. When $a<0$, at the edges of a ridge sector, $s_+\to1$ and $s_-\to s_a$; for $a\ge0$ there are no points on ridge rays. Following the branches through the sectors gives the following list of components of $\{\ol K=a\}$:
\begin{enumerate}[label=\textup{(\roman*)}]
\item $a>K_V$: in each valley sector, the arc $s_-$, with both ends at $s=0$, and the arc $s_+$, with both ends at $s=1$, since $a>0$;
\item $0\le a<K_V$: in each valley sector, two arcs, one on each side of the centre, each running from $s=0$ (along $s_-$) through $s^*$ to $s=1$ (along $s_+$);
\item $K_R<a<0$: for each edge, one arc running from $s=0$ in the adjacent valley sector, through the edge point $(s_a,\cdot)$, to $s=1$ in the adjacent ridge sector;
\item $a<K_R$: in each ridge sector, the arc $s_+$, with both ends at $s=1$, and one arc made of $s_-$ across the ridge sector, continued on both sides through the edge points into the two adjacent valley sectors and down to $s=0$.
\end{enumerate}
At $a=K_V$ and $a=K_R$ the level set consists of the four branches through each saddle of that value: two (the $s_+$ sides) ending at $s=1$ and two (the $s_-$ sides) ending at $s=0$. Near the saddle these form two smooth curves crossing transversally. Each curve contains one branch ending at $s=1$ and one ending at $s=0$, and each is either the stable or the unstable manifold of the saddle for $\ol Y$. For each regular value the list contains exactly $2k$ arcs: $k$ of each end-returning type outside $[K_R,K_V]$, and one end-to-end arc for each of the $2k$ edges inside $(K_R,K_V)$. At a critical value the four branches meet only at the indicated saddles. In particular there are no closed level curves.

\emph{Step 6: lifting to the sphere.} Off the saddles, $\ol Y$ is complete, nonvanishing and tangent to the level arcs. So each arc of Step~5 is a single orbit of $\ol Y$, traversed monotonically, and the orbit tends to the ends of the arc. The preimage of an arc is an $R$-invariant cylinder of characteristic orbits, each of which converges to $C$ or $C'$ at an end of the arc where $s\to1$ or $s\to0$. Since $C$ and $C'$ are closed orbits, the corresponding limit set is the whole circle. The open regions swept by the arcs of types (i)--(iv) are the components of (c):
\begin{itemize}
\item orbits homoclinic to $C$: the regions swept by $s_+$ for $a>K_V$ in the $k$ valley sectors and by $s_+$ for $a<K_R$ in the $k$ ridge sectors;
\item orbits homoclinic to $C'$: those swept by $s_-$ for $a>K_V$ in the valley sectors, and by the arcs of type (iv) for $a<K_R$;
\item orbits between $C$ and $C'$: for each of the $2k$ edges, the region swept by the arcs of types (ii) and (iii) adjacent to it, as $a$ runs through $(K_R,K_V)$.
\end{itemize}
Each region of the first two kinds is bounded by the branches through the saddle in its own sector. This proves (c). The tangency circles $M_j$ have $\mathscr D>0$ by Proposition~\ref{prop:model-exact}, so by Lemma~\ref{lem:meaning}(2) they are normally hyperbolic of saddle type. The projection maps $T_xL$ isomorphically onto $T_{P_j}\mathcal A$, since $T_xM_j$ is the $R$-direction. Hence the stable and unstable manifolds of $M_j$ are the preimages of those of $P_j$. Each has two cylindrical branches after removing $M_j$, one with its other end at each separatrix circle. Normal hyperbolicity gives a stable and an unstable foliation over the equilibrium circle; each orbit tending to $M_j$ therefore converges to a single point of it. This proves (b).

\emph{Step 7: proof of (d).} On $\Sigma$, $\arg z=0$, so $s=1-|w|^2$ and $\varphi=p\arg w$. The quotient map is $p$-to-one from $\Sigma\setminus\{x\}$ onto a punctured neighbourhood of the end $s=1$. Let $J$ be a closed interval contained in an open valley or ridge sector. As $s\to1$, $\ol K\to\pm\infty$ uniformly on $J$, and $s>s^*(\varphi)$ there. So $\{s>1-\delta,\ \varphi\in J\}$ lies, for small $\delta$, in a region of orbits homoclinic to $C$, namely the one in that sector. By Step~5, the regions of orbits running between $C$ and $C'$ approach $s=1$ only along the edges, in cuspidal neighbourhoods. Lifting by $\varphi=p\arg w$ gives the $2kp$ asymptotic sectors and the $2k$ classes of $p$ curvilinear sector germs cut out by the homoclinic components.

Each homoclinic region is $Y$-invariant, so $P_r$ preserves the union of its curvilinear sector germs on $\Sigma$. By Theorem~\ref{thm:char}, $P_r$ is conjugate to $h$ by a real-analytic diffeomorphism $\varpi$; in the model $\varpi$ is the projection along the leaves to $\{z=r\}$. Along leaves $d\log w/d\log z=(-\mu+cu^k)/\lambda$, and the path from $|z|=\sqrt{r^2-|w|^2}$ to $r$ is short, so $\varpi(w)=w(1+O(|w|^2))$ and $d\varpi(x)$ is the identity. Hence $dP_r(x)$ is the rotation by $-2\pi q/p$, which acts transitively on the $p$ tangent sectors of each class, since $\gcd(p,q)=1$. Thus $P_r$ permutes the corresponding curvilinear sector germs transitively; the straight rays themselves are only tangent directions and need not be invariant. Step~1 of the proof of Proposition~\ref{prop:holonomy} is exact for the model and gives $h^p(w)=w+\frac{2\pi ipc}{\lambda}r^{kq}w^{kp+1}+O(w^{kp+2})$ on $\{z=r\}$. Its attracting and repelling directions are those where $i\,w^{kp}$ is real, i.e.\ $\cos(kp\arg w)=0$. Since $d\varpi(x)=\mathrm{id}$, these are the axes of $P_r^p$. The statement at $C'$ follows by exchanging the roles of $z$ and $w$.

\emph{Step 8: the nonwandering set.} Points of $C_r$, $C'_r$ and the $M_j$ are periodic or stationary. Let $y$ be any other point. By (b) and (c), $y$ lies in a homoclinic or heteroclinic region or on a branch of some $W^u(M_j)$, or on a branch of some $W^s(M_j)$. In the first two cases the forward orbit of $y$ enters an attracting petal of $P_r^p$ at $C_r$ or $C'_r$ (Theorem~\ref{thm:char}(b) and (d)), whose saturation by the flow near the circle is an open, forward invariant set $U$ disjoint from a compact neighbourhood $K$ of $y$. By openness and compactness, $K$ may be chosen so that all its points enter $U$ in bounded time; then $\phi_t(K)\cap K=\emptyset$ for all large $t$, and $y$ is wandering. In the last case the same argument applies in backward time, with a repelling petal.\qed

\section*{Declarations}

\noindent\textbf{Funding.} 
The authors did not receive support from any organization for the submitted work.

\noindent\textbf{Competing interests.} The authors have no competing interests to declare that are relevant to the content of this article.

\noindent\textbf{Data availability.} Data sharing is not applicable to this article, as no datasets were generated or analysed. The figures were computed from explicit formulas given in the text.

\end{document}